\documentclass[12pt]{amsart}

\title{Apollonian random manifolds and their bass notes}
\author{Will Hide}
\address[Will Hide]{Mathematical Institute,
	University of Oxford,
	Andrew Wiles Building, OX2 6GG Oxford,
	United Kingdom}
\email{william.hide@maths.ox.ac.uk}

\author{Bram Petri}
\address[Bram Petri]{Institut de Math\'ematiques de Jussieu--Paris Rive Gauche and  Institut universitaire de France ; Sorbonne Universit\'e and Universit\'e Paris Cit\'e, CNRS, IMJ-PRG, F-75005 Paris, France}
\email{bram.petri@imj-prg.fr}

\author{Anna Roig Sanchis}
\address[Anna Roig Sanchis]{Laboratoire J.A. Dieudonn\'e,
Universit\'e Côte d'Azur,
CNRS,
06108 Nice,
France}
\email{anna.roig-sanchis@univ-cotedazur.fr}

\author{Joe Thomas}
\address[Joe Thomas]{Department of Mathematical Sciences, Durham University, DH1 3LE Durham, United Kingdom}
\email{joe.thomas@durham.ac.uk}
\date{\today}

\usepackage{amsmath,amsfonts,amssymb,amsthm}
\usepackage{graphicx,overpic,mathtools}
\usepackage{float,enumitem}
\usepackage{color,xcolor}
\usepackage[colorinlistoftodos]{todonotes}
\usepackage{multirow}
\usepackage{pgfplots}
\usepackage{subcaption}
\usepackage[colorlinks=true,linkcolor=red]{hyperref}
\usepackage{scalerel}
\usepackage{thmtools}
\usepackage{hyperref}
\usepackage{cleveref}
\usepackage{comment}

\pgfplotsset{compat=1.7}

\usepackage[margin=2.7cm]{geometry}  
\usepackage[indent]{parskip}

\numberwithin{equation}{section}

\newtheorem{thm}{Theorem}[section]

\newtheorem{prp}[thm]{Proposition}

\newtheorem{lem}[thm]{Lemma}

\newtheorem{innerprprep}{Proposition}
\newenvironment{prprep}[1]
  {\renewcommand\theinnerprprep{#1}\innerprprep}
  {\endinnerprprep} 

\theoremstyle{definition}
\newtheorem{dff}[thm]{Definition}

\newcommand{\nc}{\newcommand}
\nc{\dmo}{\DeclareMathOperator}
\usepackage{dsfont}

\nc{\abs}[1]{\left| #1 \right|}
\nc{\bigO}[1]{O\left(#1\right)}
\nc{\card}[1]{\left|#1\right|}
\nc{\ceil}[1]{\left\lceil #1 \right\rceil}
\nc{\CC}{\mathbb{C}}
\nc{\dilog}{\mathcal{L}}
\nc{\floor}[1]{\left\lfloor #1 \right\rfloor}
\nc{\ind}{\mathds{1}}
\nc{\ZZ}{\mathbb{Z}}
\nc{\len}[1]{\left| #1 \right|}
\nc{\littleo}[1]{o\left(#1\right)}
\dmo{\Mat}{Mat}
\nc{\NN}{\mathbb{N}}
\nc{\norm}[1]{\left|\left| #1 \right|\right|}
\nc{\QQ}{\mathbb{Q}}
\nc{\RR}{\mathbb{R}}
\nc{\st}[2]{\left\{\, #1 \,:\, #2\,\right\}}
\dmo{\supp}{supp}
\dmo{\tr}{\mathrm{tr}}
\nc{\what}{\widehat}
\dmo{\im}{Im}
\nc{\eps}{\varepsilon}
\dmo{\li}{li}
\dmo{\arccosh}{arccosh}

\dmo{\area}{area}
\dmo{\conv}{conv}
\dmo{\diam}{diam}
\dmo{\DD}{\mathbb{D}}
\dmo{\dist}{\mathrm{d}}
\nc{\HH}{\mathbb{H}}
\dmo{\Isom}{Isom}
\dmo{\MCG}{MCG}
\dmo{\MPL}{MPL}
\dmo{\Mod}{\mathcal{M}}
\dmo{\PL}{PL}
\nc{\Sphere}{\mathbb{S}}
\dmo{\sys}{sys}
\dmo{\kiss}{Kiss}
\dmo{\Teich}{\mathcal{T}}
\nc{\Torus}{\mathbb{T}}
\dmo{\vol}{vol}
\dmo{\WP}{WP}

\dmo{\convTV}{\;\stackrel{\mathrm{TV}}{\longrightarrow}\;}
\nc{\ExV}[2]{\mathbb{E}_{#1}\left[#2\right]}
\dmo{\EE}{\mathbb{E}}
\nc{\Pro}[2]{\mathbb{P}_{#1}\left[#2\right]}
\dmo{\PP}{\mathbb{P}}
\nc{\distTV}[2]{\mathrm{d}_{\rm TV}\left(#1,#2\right)}
\dmo{\UU}{\mathbb{U}}
\nc{\Var}[2]{\mathbb{V}\mathrm{ar}_{#1}\left[#2\right]}

\dmo{\alt}{\mathfrak{A}}
\dmo{\Aut}{Aut}
\dmo{\Fix}{Fix}
\dmo{\GL}{GL}
\dmo{\Hom}{Hom}
\dmo{\id}{Id}
\dmo{\PGL}{PGL}
\dmo{\PSL}{PSL}
\dmo{\PO}{PO}
\dmo{\Rep}{Rep}
\dmo{\SL}{SL}
\dmo{\SO}{SO}
\dmo{\sym}{\mathfrak{S}}
\dmo{\inv}{\mathcal{I}}
\dmo{\orb}{\mathcal{O}}
\dmo{\stab}{Stab}

\dmo{\calA}{\mathcal{A}}
\dmo{\calB}{\mathcal{B}}
\dmo{\calC}{\mathcal{C}}
\dmo{\calD}{\mathcal{D}}
\dmo{\calE}{\mathcal{E}}
\dmo{\calF}{\mathcal{F}}
\dmo{\calG}{\mathcal{G}}
\dmo{\calH}{\mathcal{H}}
\dmo{\calI}{\mathcal{I}}
\dmo{\calJ}{\mathcal{J}}
\dmo{\calK}{\mathcal{K}}
\dmo{\calL}{\mathcal{L}}
\dmo{\calM}{\mathcal{M}}
\dmo{\calN}{\mathcal{N}}
\dmo{\calO}{\mathcal{O}}
\dmo{\calP}{\mathcal{P}}
\dmo{\calQ}{\mathcal{Q}}
\dmo{\calR}{\mathcal{R}}
\dmo{\calS}{\mathcal{S}}
\dmo{\calT}{\mathcal{T}}
\dmo{\calU}{\mathcal{U}}
\dmo{\calV}{\mathcal{V}}
\dmo{\calW}{\mathcal{W}}
\dmo{\calX}{\mathcal{X}}
\dmo{\calY}{\mathcal{Y}}
\dmo{\calZ}{\mathcal{Z}}

\nc{\klav}{Klav\v{z}ar}

\nc{\bi}{\mathbf{i}}
\nc{\bj}{\mathbf{j}}
\nc{\bk}{\mathbf{k}}

\begin{document}

\begin{abstract} 
We study the spectrum of the Laplacian on two models of random hyperbolic $3$-orbifolds, related to the Apollonian group and the super Apollonian group. We determine explicit spectral gaps for these random orbifolds. Moreover, we use our model to investigate the bass note spectrum of the set of hyperbolic $3$-orbifolds.
\end{abstract}

\maketitle

\begin{figure}[ht]
\begin{center}
\includegraphics[scale=1]{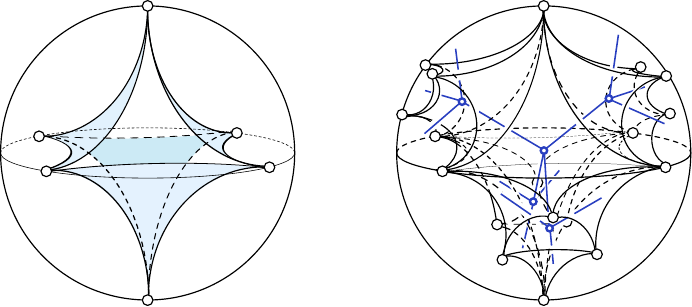}
\caption{A right-angled regular ideal octahedron and the first two generations of the octa-tree.}
\end{center}
\end{figure}

\section{Introduction}
There has been a lot of recent progress on spectral gaps of hyperbolic surfaces. In particular, it is now known that, using random constructions, one can produce surfaces of large area with a near optimal spectral gap (see for instance  \cite{HideMagee,AnantharamanMonk1,AnantharamanMonk2,CalderonMageeNaud,MageePudervanHandel,HideMaceraThomas,ShenWu}) and one can moreover approximate every real number in the interval $[0,\frac{1}{4}]$ with spectral gaps of arithmetic hyperbolic surfaces \cite{Magee_bassnotes,HidePetri}. 

In dimension three, much less is known. It has been conjectured by Magee and the fourth named author \cite[Conjecture 1.5]{MageeThomas} that there exist sequences of closed hyperbolic $3$-manifolds whose volume tends to infinity and whose spectral gap tends to that of hyperbolic $3$-space $\HH^3$ (which equals $1$).

The goal of this paper is to make a first step in dimension three. As a warm-up, we study a model for random geometrically finite $3$-manifolds of infinite volume by taking random degree $2n$ covers $S_n$ of the \textbf{Apollonian orbifold} $\Gamma_{\mathrm{Ap}}\backslash\HH^3$. It is known that the latter has no other spectrum below $1$ than the Patterson--Sullivan eigenvalue $\lambda_0(\Gamma_{\mathrm{Ap}}\backslash\HH^3)$ \cite{KelmerKontorovichLutsko}. Combining work by Bordenave--Collins \cite{BordenaveCollins} with techniques from \cite{HideMagee,Moy,BallmanMondalPolymerakis} one readily derives that, up to an $\eps$, this persists for our random covers:
\begin{thm}\label{thm_main1}
Let $\eps>0$. Then with probability tending to $1$ as $n\to\infty$, 
\[
\sigma(\Delta_{S_n}) \cap [0,1-\eps] = \{\lambda_0(\Gamma_{\mathrm{Ap}}\backslash\HH^3)\}.
\]
\end{thm}

In the theorem above $\Delta_{S_n}$ denotes the Laplacian acting on functions on $S_n$ and $\sigma(\Delta_{S_n})$ denotes its spectrum on $L^2(S_n)$. We also note that the number $\lambda_0(\Gamma_{\mathrm{Ap}}\backslash\HH^3)$ is by now known up to very high precision, due to Vytnova--Wormell \cite{VytnovaWormell}. Moreover, in this type of setting it is possible to produce subgroups with an arbitrarily small spectral gap \cite{JakobsonNaudSoares}.

Our main contribution is to the case of lattices in the isometry group $\mathrm{Isom}(\HH^3)$. We define a model for random subgroups of the \textbf{super Apollonian group} $\Gamma_{\mathrm{SA}}$: the group generated by the reflections in the faces of a regular right-angled ideal octahedron. This group admits a surjection to a fourfold free product
\[
\Gamma_{\mathrm{SA}} \to (\ZZ/2\ZZ)^{\ast 4}
\]
in which the reflections corresponding to four non-adjacent faces are sent to the generators of the four copies of $\ZZ/2\ZZ$ and the other reflections are trivialized. We will write $\Gamma_\infty = \ker(\Gamma_{\mathrm{SA}}\to (\ZZ/2\ZZ)^{\ast 4})$, a group that we will dub the \textbf{infilonian group}. 

We consider degree $2n$ covers of $M_n\to\Gamma_{\mathrm{SA}}\backslash\HH^3$ whose monodromy factors through this map (for a precise definition see Section \ref{sec_prelim}). As we will explain in Section \ref{sec_relation_models} below, this model is closely related to the model of random $3$-manifolds with boundary studied in \cite{PetriRaimbault,RoigSanchis1,RoigSanchis2}. 

The orbifolds $M_n$ have mirrors, coming from the lifts of the mirrors in $\Gamma_{\mathrm{SA}}\backslash\HH^3$. Exactly half of the mirrors  of $\Gamma_{\mathrm{SA}}$ lift to $M_n$: exactly those that correspond to the generators of $\Gamma_{\mathrm{SA}}$ that we trivialize in the surjection onto $(\ZZ/2\ZZ)^{\ast 4}$ lift. We can double $M_n$ along these mirrors, from which we obtain a double cover $DM_n\to M_n$. Note that $DM_n$ is a manifold (as opposed to an orbifold). 

The Laplacian spectrum of $DM_n$ splits into spectrum on even and spectrum on odd functions with respect to the action of the Deck group of the covering map $DM_n\to M_n$. Alternatively, we can think of this as Neumann and Dirichlet spectrum on $M_n$, in which we interpret the mirrors as boundary.  

We will prove the following result:
\begin{thm}\label{thm_main2}
Let $\eps>0$. Then as $n\to\infty$, the following hold.
\begin{enumerate}
\item For the random orbifolds $M_n$, let $\lambda_1(M_n)$ denote the first eigenvalue of their Laplacian
\[
\lambda_1(M_n) \to \lambda_0(\Gamma_\infty\backslash\HH^3),
\]
in probability.
\item\label{item_doubles} Let $\lambda_1^N(DM_n)$ and $\lambda_1^D(DM_n)$ denote the first eigenvalue of the Laplacian on even and odd functions on $DM_n$ respectively. Then
\[
\lambda_1^N(DM_n) \to \lambda_0(\Gamma_\infty\backslash\HH^3)
\]
in probability, and
\[
\lambda_1^D(DM_n) > 1- \eps
\]
with high probability.
\end{enumerate}
\end{thm}

Using complicated enough Dehn fillings, we can also use Theorem \ref{thm_main2} (\ref{item_doubles}) to produce a sequence of closed hyperbolic $3$ manifolds whose volume tends to infinity and whose first Laplacian eigenvalue tends to $\lambda_0(\Gamma_\infty\backslash\HH^3)$ (see \cite{ColboisCourtois}).

We also investigate the \textbf{bass note spectrum} of the set of hyperbolic $3$-orbifolds. We define two sets of bass notes:
\[
\mathbf{Bass}\left(\mathrm{Hyp}^3_{\mathrm{f.v.}}\right) = \st{\lambda_1(X)}{X \text{ a hyperbolic } 3\text{-orbifold of finite volume}},
\]
and
\[
\mathbf{Bass}\left(\mathrm{Hyp}^3_{\mathrm{a.}}\right) = \st{\lambda_1(X)}{X \text{ an arithmetic hyperbolic } 3\text{-orbifold}}.
\]
We observe that $\mathbf{Bass}\left(\mathrm{Hyp}_{\mathrm{a.}}\right) \subset \mathbf{Bass}\left(\mathrm{Hyp}_{\mathrm{f.v.}}\right)$. Moreover, by Mostow--Prasad rigidity, both sets above are countable. It has been conjectured by Sarnak \cite[Lecture 1]{Sarnak_Chern} that
\[
\overline{\mathbf{Bass}\left(\mathrm{Hyp}^3_{\mathrm{f.v.}}\right)} = [0,1] \cup E,
\]
where $E\subset (1,\infty)$ is an upper bounded, discrete and infinite subset. We make a first step in the direction of this conjecture:
\begin{thm}\label{thm_main3} We have
\[
[0,\lambda_0(\Gamma_\infty\backslash\HH^3)] \quad \subset \quad \overline{\mathbf{Bass}\left(\mathrm{Hyp}^3_{\mathrm{a.}}\right)}  \quad \subset \quad \overline{\mathbf{Bass}\left(\mathrm{Hyp}^3_{\mathrm{f.v.}}\right)}.
\]
\end{thm}

Finally, because it appears in many of our results, it would be nice to understand the numerical value of $\lambda_0(\Gamma_\infty\backslash\HH^3)$. We derive the following crude estimates using the Brooks--Burger transfer method \cite{Brooks_transfer,Burger} and recent work by Coulon \cite{Coulon}:
\begin{prp}\label{prp_main}
We have
\[
\frac{3-\sqrt{5+2\sqrt{3}}}{67-\sqrt{5+2\sqrt{3}}} \leq \lambda_0(\Gamma_\infty\backslash\HH^3) \leq  \delta(\Gamma_{\mathrm{Ap}})  \cdot (1- \delta(\Gamma_{\mathrm{Ap}})  /4)
\]
\end{prp}
We have $\delta(\Gamma_{\mathrm{Ap}})= 1.3056867280 4987718\ldots$ \cite{VytnovaWormell} and thus
\[
\frac{3-\sqrt{5+2\sqrt{3}}}{67-\sqrt{5+2\sqrt{3}}}  = 0.00141\ldots \quad \text{and} \quad   \delta(\Gamma_{\mathrm{Ap}})  \cdot (1- \delta(\Gamma_{\mathrm{Ap}})  /4) = 0.879482\ldots .
\]
For comparison, we note that the current record holders for the spectral gap are congruence arithmetic groups with a spectral gap of $\frac{77}{81} = 0.9506\ldots $ \cite[Theorem 4.10]{Kim}. In the process of our proof, we also improve this bound for $\Gamma(2)$. We prove that $\Gamma(2)$ has no cuspidal eigenvalues below the first Dirichlet eigenvalue of the right-angled regular ideal ocathedron (which exceeds $1$), see Proposition \ref{prp_eigenvalue_gamma_2}.

\subsection{Further notes and references}
As we will describe below, the first two theorems above rely in part on recent breakthroughs on the strong convergence of unitary representations of finitely generated groups. We will use this as a black box here. For more background on the subject we refer to the two surveys \cite{Magee_survey,vanHandel_survey}.

The question of the bass note spectrum has for instance been studied for graphs \cite{AlonWei,DongMcKenzie} hyperbolic $2$-orbifolds \cite{KravchukMazacPal}, spinors on hyperbolic surfaces \cite{AdveGiri} and $3$-manifolds \cite{BonifacioMazacPal}.

Finally, the Apollonian group and its subgroups are also a classical object of study. We refer to for instance \cite{Sarnak_Apollonian,GLMWY1,BourgainKontorovich,CFH,HaagKertzerRickardsStange} for more on this.

\subsection{Organization and proof strategy}

We start the remainder of this text with a preliminary section on the geometry and spectra of the manifolds involved in our constructions. 

After this, our first goal is to prove Theorems \ref{thm_main1} and \ref{thm_main2}. In order to control the $L^2$-spectrum of our random degree $2n$ cover (of either $\Gamma_{\mathrm{Ap}}\backslash\HH^3$ or $\Gamma_{\mathrm{SA}}\backslash\HH^3$), we need to worry about two sources of spectrum: spectrum coming from the cover corresponding to the subgroup our random covers converge to and spectrum coming from the base. In the case of $\Gamma_{\mathrm{Ap}}$ we understand both very well. Moreover, as abstract groups $\Gamma_{\mathrm{Ap}}\simeq (\ZZ/2\ZZ)^{\ast 4}$.  So if we combine the work of Bordenave--Collins \cite{BordenaveCollins} with that of either Hide--Magee \cite{HideMagee}, Moy \cite{Moy} or Ballmann--Mondal--Polymerakis \cite{BallmanMondalPolymerakis} we arrive at Theorem \ref{thm_main1} quite quickly. 

Theorem \ref{thm_main2} is more work. The same machinery based on strong convergence of permutation representations of $(\ZZ/2\ZZ)^{\ast 4}$ shows that the new eigenvalues of our random covers are close to $\lambda_0(\Gamma_\infty\backslash\HH^3)$. We however need to worry about the old eigenvalues: eigenvalues coming from $\Gamma_{\mathrm{SA}}\backslash\HH^3$. This we deal with in Section \ref{sec_spectral_input}; we show that $\Gamma_{\mathrm{SA}}$ does not have any residual spectrum and any cuspidal eigenvalues need to be at least $\frac{\pi^2}{6}$. In this same section, we prove Proposition \ref{prp_main}. 

The second half of the text is devoted to the proof of Theorem \ref{thm_main3}. The idea is to use further degree two covers of our random covers $M_n\to\Gamma_{\mathrm{SA}}\backslash\HH^3$ to obtain arbitrarily dense collections of spectral gaps in the interval $[0,\lambda_0(\Gamma_\infty\backslash\HH^3)]$. The goal is the same as in \cite{Magee_bassnotes}: we want to show that there is an elementary operation on this set of degree two covers such that:
\begin{enumerate}
\item any pair of degree two covers can be related by a finite number of moves and
\item if two covers differ by one such moves their spectral gaps are close.
\end{enumerate}
If this works, and we can produce two covers of degree two of $M_n$ whose spectral gaps lie at the endpoints of the interval $[0,\lambda_0(\Gamma_\infty\backslash\HH^3)]$, then we obtain the result we want. The elementary move we will use, consists of opening up the degree two cover $Y\to M_n$ along two ideal triangles that correspond to lifts of the same ideal triangle in the interior of $M_n$. We only use ideal triangles in the interior of $M_n$ that itself are lifts of boundary faces of the ideal octahedron. 

In order to prove that this move doesn't affect the spectral gap too much, we want to use the fact that eigenfunctions don't localize too much near the triangle in which we are performing the move. Indeed, if this is the case, we can bump the function off to $0$ along the triangle at little cost to its Rayleigh quotient. In Magee's paper, this is based on tangle-freeness combined with ideas from \cite{Gamburd,GLMST}. However, our orbifolds $M_n$ are not tangle-free. They locally look like $\Gamma_\infty\backslash\HH^3$ rather than $\HH^3$. So we introduce a notion that we dub $\Gamma_\infty$-tangle-freeness. Most of the work of this part of the paper goes into showing that, below the bass note of $\Gamma_\infty\backslash\HH^3$, this suffices for the delocalization properties we need.

\subsection*{Acknowledgements}

We thank Simon André, Fathi Ben Aribi, Claire Burrin, Rémi Coulon, Michael Magee and Frédéric Naud for useful conversations. 

\subsection*{Funding}
JT has received funding from the Leverhulme Trust through a Leverhulme Early Career Fellowship (Grant No. ECF-2024-440). BP was partially supported by the grant ANR-23-CE40-0020-02 ``MOST''.

\section{Preliminaries}\label{sec_prelim}

In this section, we gather some of the basic facts on the geometry and spectra of various hyperbolic orbifolds that will play a role in the rest of this paper. For a proper introduction, we refer to for instance \cite{BenedettiPetronio,Martelli} for the geometry and for example \cite{Iwaniec,CohenSarnak,ElstrodtGrunewaldMennicke} for the spectral theory.

\subsection{The geometry of hyperbolic $3$-space}

All of our calculations will use the upper half space model for hyperbolic $3$-space:
\[
\HH^3 = \st{(x,y,t) \in \RR^3}{t>0},\quad ds^2 = \frac{dx^2+dy^2+dt^2}{t^2}.
\]
We will also often identify the first two coordinates with a copy of the complex plane and write $(z,t)$ for a point in $\HH^3$. 

The isometry group $\mathrm{Isom}(\HH^3)$ of $\HH^3$ can be identified with $\PSL(2,\CC)\rtimes (\ZZ/2\ZZ)$. Here $\PSL(2,\CC)$ is the subgroup of orientation preserving isometries and acts on $\HH^3$ by 
\[
\left[
\begin{array}{cc}
a & b \\
c & d
\end{array}
\right]\cdot (z,t) = \left(\frac{(az+b)\cdot (\overline{cz+d}) + a \overline{c} t^2}{\abs{cz+d}^2+\abs{c}^2t^2},\; \frac{t}{\abs{cz+d}^2+\abs{c}^2t^2} \right)
\]
and, if $\rho$ is the generator for the $\ZZ/2\ZZ$ factor, then 
\[
\rho(z,t) = (\overline{z},t),
\]
which also implies that in the semi-direct product structure, $\rho$ acts on $\PSL(2,\CC)$ by complex conjugation. The group $\PSL(2,\CC)$ is isomorphic to $\PGL(2,\CC)$. The analogous fact is false over many other rings, like $\RR$, and notably in our setting: $\ZZ[i]$.

Given a discrete subgroup $\Gamma < \mathrm{Isom}(\HH^3)$, the quotient $\Gamma\backslash\HH^3$ has the structure of a hyperbolic orbifold. All the groups that we will consider are virtually torsion free, which means that $\Gamma\backslash\HH^3$ has a finite degree cover that is a manifold.

\subsection{The geometry of cusps}\label{sec_geometry_of_cusps}

Most manifolds that we will discuss have cusps. We will sometimes need to normalize certain parameters corresponding to (virtually) parabolic subgroups of the Kleinian groups that appear. We will now explain how we do this. 

Throughout this part of the text, we will assume that $\Gamma_0<\mathrm{Isom}(\HH^3)$ is a lattice and $\Lambda\triangleleft \Gamma_0$ is some normal subgroup. Both are considered fixed throughout. Moreover, $\Gamma<\mathrm{Isom}(\HH^3)$ is a lattice (that does vary) such that $\Lambda<\Gamma<\Gamma_0$. Because $\Gamma$ is a lattice, it is of finite index inside $\Gamma_0$.

Because $\Gamma_0$ is a lattice, the cusps of $\Gamma_0$ are all of rank two. That is, if $\mathbf{a}\in\partial_\infty\HH^3$ is a fixed point of some parabolic element of $\Gamma_0$, then the group $\Big(\Gamma_0\Big)_{\mathbf{a}}$ of elements of $\Gamma_0$ that fix $\mathbf{a}$ contains a subgroup of finite index that is abstractly isomorphic to $\ZZ^2$ and in which every non-trivial element is parabolic. Moreover, again because $\Gamma_0$ is a lattice, it has finitely many cusps ($\Gamma_0$-orbits of parabolic fixed points). 

We will once and for all fix a set of representatives $\mathbf{A}=\{\mathbf{a}_1,\ldots\mathbf{a}_m\}\subset \partial_\infty\HH^3$ for the cusps of $\Gamma_0$. For every $\mathbf{a}\in\mathbf{A}$, we will also once and for all fix a Möbius transformation $h_{\mathbf{a}}\in\PSL(2,\CC)$ such that 
\[
h_{\mathbf{a}}(\mathbf{a})=\infty.
\]
This also allows us to define the \textbf{height} of any point $x\in\HH^3$ with respect to the cusp $\mathbf{a}$:
\[
t_{\mathbf{a}}(x) = \max\st{t(h_{\mathbf{a}}\cdot\gamma\cdot x)}{\gamma\in\Gamma_0},
\]
where $t:\HH^3\to\RR$ denotes the vertical coordinate in the upper half-space model for $\HH ^3$. Another way of phrasing this is to say that we measure the height of
\[
[x] \quad\in \quad h_{\mathbf{a}}\cdot\Big(\Gamma_0\Big)_{\mathbf{a}}\cdot h_{\mathbf{a}}^{-1} \; \Big\backslash \; \HH^3,
\] 
which is a cusp obtained by quotienting $\HH^3$ by a \textbf{standard} parabolic group 
\[
h_{\mathbf{a}}\cdot\Big(\Gamma_0\Big)_{\mathbf{a}}\cdot h_{\mathbf{a}}^{-1} = \left\langle \left[\begin{array}{cc} 1 & w_1 \\ 0 & 1 \end{array}\right], \left[\begin{array}{cc} 1 & w_2 \\ 0 & 1 \end{array}\right]\right\rangle
\]
where $w_1,w_2\in\CC$ are linearly independent over $\RR$.

With respect to $\Gamma$, the cusps of $\Gamma_0$ split into finitely many equivalence classes. Because we're not assuming $\Gamma$ is normal in $\Gamma_0$, the shapes of these cusps may vary, even when they correspond to the same cusp in $\Gamma_0$. If $\mathbf{b}$ is a cusp of $\Gamma$, that lies in the $\Gamma_0$-equivalence class of the cusp $\mathbf{a}$ of $\Gamma_0$, we define the height with respect to $\mathbf{b}$ as
\[
t_{\mathbf{b}}(x) = \max\st{t(h_{\mathbf{a}}\cdot\gamma_\mathbf{b}\cdot \gamma\cdot x)}{\gamma\in\Gamma},
\]
where $\gamma_{\mathbf{b}}\in\Gamma_0$ is any element such that $\gamma_{\mathbf{b}}(\mathbf{b}) = \mathbf{a}$ (one easily checks that the choice of $\gamma_{\mathbf{b}}$ is of no influence), which can again be thought of as measuring the height in a standard model for the cusp. The \textbf{area} of a cusp $\mathbf{b}$ of $\Gamma$ is defined as
\[
\mathrm{area}_{\Gamma}(\mathbf{b}) = \mathrm{area}\left(\Gamma_{\mathbf{b}}\backslash\st{x\in \HH^3}{t_{\mathbf{b}}(x)=1}\right).
\]

Because we're not assuming that $\Lambda<\Gamma_0$ is of finite index, multiple things can happen to a cusp $\mathbf{a}$ of $\Gamma_0$. Indeed, inside $\Lambda$, the cusp $\mathbf{a}$ might either disappear (if none of the parabolics fixing it are contained in $\Lambda$) or it might lift to a rank one or rank two cusp. Note that the fact that $\Lambda\triangleleft\Gamma_0$ implies that what happens is uniform over the lifts of $\mathbf{a}$. We will write $\mathrm{rank}_{\Lambda}(\mathbf{a})\in\{0,1,2\}$ for the rank of $\mathbf{a}$ with respect to $\Lambda$. Note furthermore that this expression also makes sense for cusps of $\Gamma$. Indeed, every cusp $\mathbf{b}$ of $\Gamma$ corresponds to a unique cusp $\mathbf{a}$ of $\Gamma_0$ and thus we may write  $\mathrm{rank}_{\Lambda}(\mathbf{b})=\mathrm{rank}_{\Lambda}(\mathbf{a})$. 

Finally, we note that if $\mathrm{rank}_{\Lambda}(\mathbf{a})=2$, then $\Lambda_{\mathbf{a}}<\Big(\Gamma_0\Big)_{\mathbf{a}}$ is of finite index and the area $\area_{\Lambda}(\mathbf{a})$ is finite. If $\mathrm{rank}_{\Lambda}(\mathbf{a})=1$, then the corresponding area is no longer finite. However, we can still speak of the \textbf{length} $\mathrm{length}_{\Lambda}(\mathbf{a})$ of the cusp $\mathbf{a}$. This is the minimal positive translation length of an element of $\Lambda_{\mathbf{a}}$ on the Euclidean plane $\st{x\in\HH^3}{t_{\mathbf{a}}(x)=1}$.

\subsection{Spectral theory}

Now we need some basic facts on the spectral theory of $3$-manifolds with cusps. Suppose that $\Gamma<\mathrm{Isom}(\HH^3)$ is a non-uniform lattice. In this case, the spectrum of the Laplacian $L^2(\Gamma\backslash\HH^3)$ splits into continuous and discrete spectrum. The former comes from generalized Eisenstein series and the latter from cusp forms and residues of poles of Eisenstein series. The last of these is called \textbf{residual spectrum}. 

More concretely, to each cusp, thought of as the $\Gamma$-orbit of a parabolic fixed point $\mathbf{a}\in\partial\HH^3$, we can associate an \textbf{Eisenstein series}
\[
E_{\mathbf{a}}(p,s) = \sum_{\gamma\in \Gamma_{\mathbf{a}}'\backslash \Gamma} t(h_{\mathbf{a}}\cdot \gamma\cdot p)^s,\quad p\in\HH^3,\mathrm{Re}(s) > 2,
\]
where, as above, $\Gamma_{\mathbf{a}}$ denotes stabilizer of $\mathbf{a}$ in $\Gamma$ and $\Gamma_{\mathbf{a}}'$ denotes the subgroup generated by the parabolic elements in $\Gamma_{\mathbf{a}}$. Likewise, $h_{\mathbf{a}}\in\PSL(2,\CC)$ still is some Möbius transformation such that $h_{\mathbf{a}}(\mathbf{a})=\infty$. As a function of $s$, this series can be meromorphically continued to a function that has no poles, except for a finite number of simple poles in the interval $(1,2]$. The pole at $2$ corresponds to constant functions. The other poles give rise to discrete spectrum in the interval $(0,1)$.

In each cusp we can write a Fourier expansion
\[
E_{\mathbf{a}}(h_{\mathbf{b}}^{-1}p,s) = \delta_{\mathbf{ab}}\cdot t(p)^s + \phi_{\mathbf{ab}}(s) \cdot t(p)^{2-s} + g_{\mathbf{ab}}(p,s),
\]
where $g_{\mathbf{ab}}(p,s)$ consists of all the non-zero Fourier coefficients and decays exponentially as a function of $t(p)$. Moreover, if $E_{\mathbf{a}}$ has a pole at $s\in (1,2)$, then one of the $\phi_{\mathbf{ab}}(s)$ has a pole. The matrix 
\[
\Phi(s) = \Big(\phi_{\mathbf{ab}}(s)\Big)_{\mathbf{ab}},
\]
where $\mathbf{a}$ and $\mathbf{b}$ run over a full set of representatives of the cusps, is called the \textbf{scattering matrix} of $\Gamma$.

\subsection{The super Apollonian group}\label{sec_SA}

Let us now formally define the first of three groups that appear in our constructions: the super Apollonian group $\Gamma_{\mathrm{SA}}$. As an abstract group, this is the right-angled Coxeter group associated to the $1$-skeleton of the $3$-cube (see Figure \ref{pic_cube}). 
\begin{figure}[h]
\begin{center}
\begin{overpic}{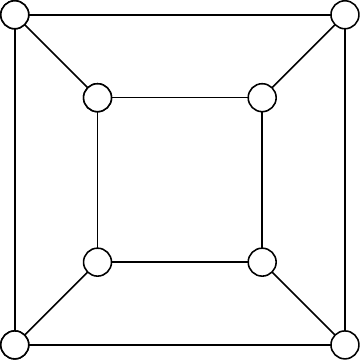}
\put(25,17){$r_1$}
\put(70,17){$r_3^\perp$}
\put(25,79){$r_4^\perp$}
\put(70,79){$r_2$}
\put(11,7){$r_2^\perp$}
\put(83,7){$r_4$}
\put(12,89){$r_3$}
\put(81,89){$r_1^\perp$}
\end{overpic}
\caption{The defining graph of $\Gamma_{\mathrm{SA}}$.}\label{pic_cube}
\end{center}
\end{figure}

Concretely, if $\calG=(V,E)$ is the $1$-skeleton of the $3$-cube, then
\[
\Gamma_{\mathrm{SA}} \simeq \langle v\in V|\; [v,w]=e \; \forall \{v,w\} \in E \text{ and } v^2=e\;\forall v\in V\rangle
\]
The elements of $V=\{r_1,r_2,r_3,r_4,r_1^\perp,r_2^\perp,r_3^\perp,r_4^\perp\}$ will be called the \textbf{standard generators} of $\Gamma_{\mathrm{SA}}$. To understand the action of $\Gamma_{\mathrm{SA}}$ on $\HH^3$ we need to describe an embedding of $\Gamma_{\mathrm{SA}}$ into $\mathrm{Isom}(\HH^3)$. In fact, $\Gamma_{\mathrm{SA}}$ can be seen as a subgroup of $\PSL(2,\ZZ[i])\rtimes(\ZZ/2\ZZ)<\PGL(2,\ZZ[i])\rtimes(\ZZ/2\ZZ)$. Because of the difference between $\PGL(2,\ZZ[i])$ and $\PSL(2,\ZZ[i])$ (the latter is of index $2$ in the former) and in particular their congruence subgroup of level $2$, we will use $\PGL(2,\ZZ[i])\rtimes(\ZZ/2\ZZ)$ here.

 The standard generators embed as follows:

\[
r_1 =\left( \left[\begin{array}{cc}
1+2i & -2 \\
2 & -1+2i
\end{array}
 \right], \rho\right),
\quad 
r_1^\perp = \left( \left[\begin{array}{cc}
 1 & 0 \\
 0 & 1
\end{array}
 \right], \rho\right),
\]
\[
r_2= \left( \left[\begin{array}{cc}
1 & 0 \\
2 & -1
\end{array}
 \right], \rho\right),
\quad 
r_2^\perp = \left( \left[\begin{array}{cc}
 1 & 2i \\ 
 0 & 1
\end{array}
 \right], \rho\right)
\]
\[
r_3= \left( \left[\begin{array}{cc}
-1 & 2 \\
0 & 1
\end{array}
 \right], \rho\right),
\quad 
r_3^\perp=  \left( \left[\begin{array}{cc}
 1 & 0 \\ 
 -2i & 1
\end{array}
 \right], \rho\right)
\]
\[
r_4= \left( \left[\begin{array}{cc}
-1 & 0 \\
0 & 1
\end{array}
 \right], \rho\right),
\quad 
r_4^\perp=  \left( \left[\begin{array}{cc}
 1-2i & 2i \\
 -2i & 1+2i
\end{array}
 \right], \rho\right)
\]
see for instance \cite[p. 2298]{CFH}. As Chaubey--Fuchs--Hines--Stange note,
\[
\Gamma_{\mathrm{SA}} = \Gamma(2) \rtimes \langle \rho \rangle,
\]
where $\Gamma(2)$ denotes the level $2$ congruence subgroup of $\PGL(2,\ZZ[i])$ (see Section \ref{sec_residual_spectrum} for a precise definition). We note that the level $2$ congruence subgroup of $\PSL(2,\ZZ[i])$ is of index $2$ in $\Gamma(2)$, which is the reason we work in $\PGL(2,\ZZ[i])$ instead of $\PSL(2,\ZZ[i])$.

The above also means that $\Gamma(2) = \ker\left( \mathrm{or}: \Gamma_{\mathrm{SA}} \to \ZZ/2\ZZ\right)$, where $ \mathrm{or}: \Isom^+(\HH^3) \to \ZZ/2\ZZ$ denotes the orientation homomorphism.

\subsection{The regular right-angled ideal octahedron}

The generators of $\Gamma_{\mathrm{SA}}$ above are exactly the reflections in the faces of a \textbf{regular right-angled ideal octahedron} $O\subset\HH^3$ in standard position: the convex hull in the upper half space model for $\HH^3$ of the vertices $\{0,1,i,\infty,i+1,(i+1)/2\} \subset \partial_\infty\HH^3$. The fact that $O$ is right-angled implies that the group generated by the reflections in its faces indeed has the structure of the right-angled Coxeter group defined by the $1$-skeleton of its dual polytope (the cube). As a result the orbifold $\Gamma_{\mathrm{SA}}\backslash\HH^3$ can be identified with a copy of $O$ in which the faces are interpreted as mirrors.

We call this octahedron $O$ regular because it has octahedral symmetry. Indeed its group of orientation preserving isometries is the group $G=\mathrm{Isom}^+(O)$ generated by  the Möbius transformations
\[
R_3 = \left[ \begin{array}{cc}
1 & -i \\
-i & 0
\end{array}\right] \quad \text{and} \quad R_4 =  \frac{1}{\sqrt{2}}
\left[\begin{array}{cc}
1-i & -1+i \\
0 & 1+i
\end{array}
\right]
\]
of orders three and four respectively. As an abstract group, $G$ is a copy of the symmetric group $\sym_4$. The \textbf{fundamental corner} 
\[
\calC =\st{ (x,y,t)\in\HH^3}{ 0\leq x,y \leq \frac{1}{2},\; x^2+y^2+t^2\geq 1} \subset O.
\]
is a fundamental domain for the action of $G$ on $O$ (see Figure \ref{pic_fund_corner_2} for a picture), which in particular means that we need $24$ copies of $\calC$ to tile $O$. Each face of $O$ is incident to $6$ of these copies of $\calC$. On a side note, we also observe that $\calC$ is half of a fundamental domain for $\PSL(2,\ZZ[i])$ acting on $\HH^3$ (see for instance \cite[Proposition 7.3.9]{ElstrodtGrunewaldMennicke}). Indeed, $\calC$ is a fundamental domain for $\PSL(2,\ZZ[i])\rtimes \langle\rho\rangle$ acting on $\HH^3$.

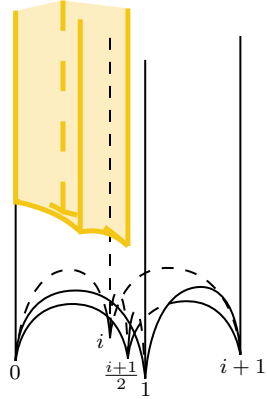
\begin{figure}[h]
  \begin{center}
     \begin{overpic}{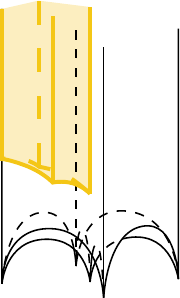}
     \put(-1,0){\tiny{$0$}}
     \put(22,8){\tiny{$i$}}
     \put(23.5,-0.5){\tiny{$\frac{i+1}{2}$}}
     \put(33.5,-4.5){\tiny{$1$}}
     \put(54.5,1.5){\tiny{$i+1$}}
     \end{overpic}
	\caption{The fundamental corner $\calC$.}        \label{pic_fund_corner_2}
	\end{center}
\end{figure}

\subsection{The Apollonian group}
We will think of the Apollonian group $\Gamma_{\mathrm{Ap}}<\Gamma_{\mathrm{SA}}$ as the subgroup generated by reflections in $4$-non-adjacent faces of $O$. In terms of the generators above, we can for instance write $\Gamma_{\mathrm{Ap}} = \langle r_1, r_2, r_3, r_4\rangle$. As an abstract group $\Gamma_{\mathrm{Ap}}\simeq (\ZZ/2\ZZ)^{\ast 4}$. It's called the Apollonian group because its limit set is an Apollonian gasket. 

A fundamental domain for $\Gamma_{\mathrm{Ap}}$ can be obtained as follows. The planes in which the generators reflect, cut $\HH^3$ up into $5$ connected components. The unique component of these that is incident to all $4$ of the reflective planes is a fundamental domain. The quotient of $\Gamma_{\mathrm{Ap}}\backslash\HH^3$ can be thought of as this fundamental domain in which the faces are replaced with mirrors.

\subsection{The infilonian group}\label{sec_infilonian}
In this section we discuss the third group we need. We once and for all fix a surjection
\[
\pi:\Gamma_{\mathrm{SA}} \longrightarrow (\ZZ/2\ZZ)^{\ast 4} = \langle t_1,t_2,t_3,t_4|\; t_1^2=t_2^2=t_3^2=t_4^2=e\rangle
\]
defined by
\[
r_1^\perp\mapsto t_1,\; r_2^\perp\mapsto t_2,\; r_3^\perp\mapsto t_3,\;r_4^\perp\mapsto t_4 
\]
and
\[
r_1\mapsto e,\; r_2\mapsto e,\; r_3\mapsto e,\;r_4\mapsto e. 
\]
Our group $\Gamma_\infty$ will be the kernel of this map. This means that
\[
\Gamma_{\mathrm{Ap}}<\Gamma_\infty<\Gamma_{\mathrm{SA}}.
\]
In fact $\Gamma_\infty$ is the normal closure of $\Gamma_{\mathrm{Ap}}$ in $\Gamma_{\mathrm{SA}}$.

\subsection{The octa-tree}

Let us now describe the geometry of the quotient $\Gamma_\infty\backslash\HH^3$. The cosets of $\Gamma_\infty<\Gamma_{\mathrm{SA}}$ are labelled by elements of $(\ZZ/2\ZZ)^{\ast 4}$, which, through our surjection, we can also identify with elements of the group $\Gamma_{\mathrm{Ap}}^\perp=\langle r_1^\perp,r_2^\perp,r_3^\perp,r_4^\perp\rangle < \Gamma_{\mathrm{SA}}$. The latter is another copy of the Apollonian group, not only as an abstract group, but also as a Kleinian group. We obtain a fundamental domain for $\Gamma_\infty$ by setting
\[
T = \Gamma_{\mathrm{Ap}}^\perp\cdot O.
\]
Geometrically, $T$ can be obtained by gluing countably many copies of $O$ in the pattern of a $4$-valent tree (using $4$ pairwise non-adjacent faces per copy of $O$ for the gluings), which is why will call $T$ the \textbf{octa-tree}. The boundary of $T$ consists of countably many totally geodesic hyperbolic planes, whose boundaries form an Apollonian gasket in $\partial_\infty\HH^3$. The group $\Gamma_\infty$ is generated by the reflections in these planes and as such is isomorphic to 
\[
 \Big(\ZZ/2\ZZ\Big)^{\ast \NN} = \langle t_1,t_2,\ldots |\; 1 = t_1^2=t_2^2=\ldots \rangle.
\]
The quotient $\Gamma_\infty\backslash\HH^3$ can be identified with a copy of $T$ in which the boundary faces have been replaced by mirrors.

\subsection{The model}\label{sec_model}

Now that we have introduced all our groups and spaces it's time to define our models of random covers.

To do so, we let $\sym_{2n}$ denote the symmetric group on $2n$ letters and 
\[
\varphi_n \in \Hom\Big((\ZZ/2\ZZ)^{\ast 4},\sym_{2n}\Big)
\]
be a homomorphism that is obtained by sending every generator of $(\ZZ/2\ZZ)^{\ast 4}$ to a uniformly random \textbf{perfect matching}: a fixed point free involution on $\{1,\ldots,2n\}$. Note that, because $\Gamma_{\mathrm{Ap}} \simeq (\ZZ/2\ZZ)^{\ast 4}$, $\varphi_n$ can be thought of as an element of $\Hom(\Gamma_{\mathrm{Ap}},\sym_{2n})$ as well.

Using $\varphi_n$, we build two random orbifolds: $S_n$ and $M_n$. The first of these, $S_n$, is the orbifold of infinite volume, obtained as the cover of $\Gamma_{\mathrm{Ap}}\backslash \HH^3$ corresponding to $\varphi_n$. That is, we set
\[
S_n = \Gamma_{\mathrm{Ap}} \backslash \Big(\HH^3 \times \{1,\ldots, 2n\}\Big),
\]
where $\Gamma_{\mathrm{Ap}}$ acts through $\varphi_n$ on the second factor. Note that the fact that we're using perfect matchings implies that $S_n$ is a manifold for all $n\geq 1$. Indeed, all the mirrors of the Apollonian orbifold $\Gamma_{\mathrm{Ap}}\backslash\HH^3$ lift to hyperplanes in the interior of $S_n$.

The second, that we will call $M_n$, is the cover of $\Gamma_{\mathrm{SA}}\backslash\HH^3$ corredponding to $\varphi_n\circ\pi$, where $\pi:\Gamma_{\mathrm{SA}}\to (\ZZ/2\ZZ)^{\ast 4}$ is the surjection that we fixed in Section \ref{sec_infilonian}. So, 
\[
M_n = \Gamma_{\mathrm{SA}} \backslash \Big(\HH^3 \times \{1,\ldots, 2n\}\Big),
\]
where $\Gamma_{\mathrm{SA}}$ acts through $\varphi_n\circ\pi$ on the second factor. As we noted in the introduction, the manifolds $M_n$ form the main object of study of this paper.

Let us descibe some consequences of our set-up. Because we factor through $\pi$, this means that the mirrors of $\Gamma_{\mathrm{SA}}\backslash\HH^3$ corresponding to reflections that $\pi$ trivializes (so $\{r_1,r_2,r_3,r_4\}$) all lift to $n$ mirrors of $M_n$. On the other hand, again because we are using perfect matchings, the other mirrors of the super Apollonian orbifold lift to ideal triangles in the interior of $M_n$. We will call these ideal triangles the \textbf{interior faces} of $M_n$. We will call any orbifold that can be obtained through this procedure an \textbf{octahedral orbifold}.

\subsection{The relation to random manifolds with boundary}\label{sec_relation_models}

In \cite{PetriRaimbault}, a model for random manifolds with boundary was introduced. In this model, truncated tetrahedra are glued together along their hexagonal faces according to the combinatorics of a random $4$-regular simple graph. As observed in that paper, these manifolds can be thought of as Dehn fillings of manifolds $Y_{2n}$ obtained by gluing copies of $O$ along $4$ non-adjacent faces (in \cite{PetriRaimbault}, the number of copies of $O$ does not need to be even, but in order to compare, we need to restrict to the even case). Let us briefly discuss what the difference is between the manifolds $Y_{2n}$ and the manifolds $M_n$ of the current paper.

First of all, the underlying graph structure of the two models is different (configurations versus perfect matchings). However, if we condition on there being no loops or multiple edges, then these models are contiguous: sequences of events hold with high probability in one model if and only if they hold with high probability in the other \cite{RobinsonWormald,Janson}. This conditioning is done in \cite{PetriRaimbault} when constructing $Y_{2n}$, but in this paper we don't condition on the graph not having multiple edges. Nonetheless, we could do so here and it wouldn't change our results. 

The main difference between the models is a slight geometric difference: in \cite{PetriRaimbault}, the underlying graph does not define the manifold. Indeed, there each edge is decorated with an element of $\ZZ/3\ZZ$ that describes a choice of one of three possible gluings between the two faces. In this paper, we just glue with the identity.

\section{Some spectral input}\label{sec_spectral_input}

In this section we investigate the spectral theory of various polytopes and orbifolds that play a role in this paper.

\subsection{Residual spectrum}\label{sec_residual_spectrum}

We start with the residual spectrum of certain congruence groups. We will use this only in the case of the level $2$ congruence group. But, since the proof doesn't change if we replace the $2$ by an $N$, we write $N$.

We will prove the fact that for a certain congruence subgroup of $\PGL(2,\ZZ[i])$ there is no non-trivial residual spectrum. This is potentially well known, but we couldn't find it in the literature explicitly (for $\PGL(2,\ZZ[i])$ itself it's implied by \cite[Example 1.14]{CohenSarnak} and \cite[Corollary 2.3]{ElstrodtGrunewaldMennicke_Eisenstein}). As such, we will write down a proof, following standard methods.

For $N \in \NN$, we will write $(N)$ for the principal ideal in $\ZZ[i]$ generated by $N$ and define the level $N$ congruence subgroup of $\PGL(2,\ZZ[i])$ as: 
\[
\Gamma(N) = \ker\Big(\PGL(2,\ZZ[i]) \to \PGL(2,\ZZ[i]/(N))\Big) \]
where the map is the reduction of coefficients modulo $(N)$. 

\begin{prp}\label{prp_residual_spectrum}
For any $N\in\NN$, the only residual spectrum of $\Gamma(N)$ corresponds to constant functions. 
\end{prp}

This comes down to proving that the scattering matrix $\Phi(s)$ has no poles in the interval $(1,2)$. We will follow one of the strategies explained in \cite[Section 11.2]{Iwaniec} (which deals with the congruence subgroups of $\PSL(2,\ZZ)$). The proof consists of two lemmas. The first of these is a $3$-dimensional version of \cite[Theorem 6.9]{Iwaniec}:

\begin{lem}
Let $\Gamma < \PGL(2,\CC)$ be a non-uniform lattice with scattering matrix $\Phi = \Big(\phi_{\mathbf{ab}}\Big)_{\mathbf{ab}}$. If for some $\mathbf{a}$, $\mathbf{b}$ the entry $\phi_{\mathbf{ab}}$ has a pole at $s\in(1,2)$ then $\phi_{\mathbf{aa}}$ has a pole at $s$ as well.
\end{lem}

\begin{proof}
We prove this in the same manner as in \cite[Theorem 6.9]{Iwaniec}, based on the Maa\ss--Selberg relations. First, given a parameter $A>0$, we define the truncated Eisenstein series
\[
E^A_{\mathbf{a}}(p,s) = \left\{
\begin{array}{ll}
E_{\mathbf{a}}(p,s) & \text{if } t(h_{\mathbf{b}}^{-1}p) \leq A \text{ for all cusps } \mathbf{b}, \\[2mm]
E_{\mathbf{a}}(p,s)-\delta_{\mathbf{ab}}\cdot t(p)^s - \phi_{\mathbf{ab}}(s)\cdot t(p)^{2-s} & \text{if } t(h_{\mathbf{b}}^{-1}p) > A \text{ for the cusp } \mathbf{b}.
\end{array}
\right.
\]
we will assume that $A$ is chosen so that the cusp neighborhoods $\{t(h_{\mathbf{b}}^{-1}p)>A\}$ project to standard horoball neighborhoods in $\Gamma\backslash\HH^3$ that are pairwise disjoint.

Now a special case of the Maa\ss--Selberg relations (which follow from an application of Stokes' theorem, see for instance \cite[Theorem 3.3.6]{ElstrodtGrunewaldMennicke}) yields, for $\sigma>1$ and $v>0$,
\begin{multline*}
\int_{\calF} \abs{E^A_{\mathbf{a}}(p,\sigma+iv)}^2\; d\mu(p) + \frac{1}{2\sigma-2} A^{2-2\sigma} \sum_{\mathbf{b}} \abs{\phi_{\mathbf{ab}}(\sigma+iv)}^2 \\
= \frac{1}{2\sigma-2} A^{2\sigma-2} -\frac{1}{v} \mathrm{Im}\Big(\phi_{\mathbf{aa}}(\sigma+iv)  A^{-2iv}\Big),
\end{multline*}
where $\calF$ denotes a fundamental domain for $\Gamma$ in $\HH^3$. At first sight, the factor $1/v$ on the right hand side might seem suspicious, but $\phi_{\mathbf{aa}}$ is real on the real line, so at regular points on the real line, the expression on the right hand side does not blow up.

We now observe that the left hand side of this equation is non-negative. This means in particular that, if $\phi_{\mathbf{ab}}(\sigma+iv)$ blows up as $v\to 0$, the point $\sigma$ cannot be a regular point of $\phi_{\mathbf{aa}}$.
\end{proof}

The upshot of the previous lemma is that we only need to prove that the diagonal coefficients of the scattering matrix have no poles in the interval $(1,2)$. This we do by explicitly computing them.

\begin{lem}
Let $N\in\NN$. Then we have
\[
\phi_{\mathbf{aa}}(s) =  \frac{\pi}{4 \cdot (s-1)} \frac{\zeta_{\QQ(i)}(s-1)}{\zeta_{\QQ(i)}(s)} \cdot \frac{1}{N^{2s+2}} \cdot \prod_{p|N} \frac{1}{1-p^{-2s}},
\]
where $\zeta_{\QQ(i)}$ denotes the Dedekind zeta function for $\QQ(i)$.
\end{lem}

\begin{proof}
We will perform a similar computation to that in \cite[Example 1.14]{CohenSarnak}. First of all, we note that $\Gamma=\PGL(2,\ZZ[i])$ has only one cusp: it is proved in for instance \cite[Proposition 7.3.9]{ElstrodtGrunewaldMennicke} for $\PSL(2,\ZZ[i])$ and since this is a subgroup of $\Gamma$, that proves it for the latter as well. The cusp corresponds to the $\Gamma$-orbit of $\infty \in\partial\HH^3$. As such, $\Gamma(N)_{\mathbf{a}} = h_{\mathbf{a}}^{-1}\cdot \Gamma(N)_\infty \cdot h_{\mathbf{a}}$ for any cusp $\mathbf{a}$ of $\Gamma(N)$. This implies that
\[
E_a(h_{\mathbf{a}}^{-1}p,s) = E_\infty(p,s)
\]
for all cusps $\mathbf{a}$ of $\Gamma(N)$.

Now we observe that if $c,d\in\ZZ[i]-\{0\}$ with $(c,d)=(1)$ and
\[
\left[\begin{array}{cc}
a & b \\
c & d
\end{array}\right] \in \Gamma(N) \quad \text{and} 
\quad
\left[\begin{array}{cc}
a' & b' \\
c & d
\end{array}\right] 
\]
then there exists some $m \in (N)$ such that 
\[
a-a' = m\cdot c \quad \text{and} \quad b-b' = m\cdot d.
\]
Indeed, both $ad-bc$ and $a'd-b'c$ are units in $\ZZ[i]$ and hence
\[ 
(a-a')\cdot d = u\cdot (b-b')\cdot c,
\]
where $u\in\ZZ[i]$ is a unit. Because $\ZZ[i]$ is a unique factorization domain and $(c,d)=(1)$, there exists some $m\in\ZZ[i]$ with $a-a'=u\cdot m\cdot c$ and $b-b'=m\cdot d$. Moreover, because $b,b'\in (N)$ and $d = 1\mod (N)$, we obtain that $m \in (N)$.

Likewise, using the fact that the only units in $\ZZ[i]$ are the fourth roots on unity, we obtain the same statement in the case when $d=0$ (which can only happen when $N=1$). In other words, non-trivial cosets correspond to a choice of $c$ and $d$.

This means that, if we write $p=(z,t)$, we get
\begin{align*}
E_\infty(p,s) & = t^s+\frac{1}{4}\sum_{\substack{c \in (N),\; c\neq 0 ,\\ d\in 1 + (N),\; (c,d) = (1)}} \frac{t^s}{(\abs{cz+d}^2+\abs{c}^2t^2)^s} \\
& = t^s + \frac{1}{4}\sum_{c \in (N) - \{0\}} \frac{t^s}{\abs{c}^{2s}}\sum_{d\in 1 + (N),\; (c,d) = (1)} \frac{1}{(\abs{z+d/c}^2+t^2)^s},
\end{align*}
where the factor $\frac{1}{4}$ comes from the fact that we're considering matrices up to a unit in $\ZZ[i]$.

Next, we rewrite the inner sum using the right action of 
\[
\Gamma(N)_\infty' = \st{\left[\begin{array}{cc}
1 & m \\
0 & 1
\end{array}\right]}{m\in (N)}
\]
on $\Gamma(N)_\infty'\backslash\Gamma(N)$. We have
\[
\left[\begin{array}{cc}
\ast & \ast  \\
c & d
\end{array}\right]
\cdot 
\left[\begin{array}{cc}
1 & m \\
0 & 1
\end{array}\right]
=
\left[\begin{array}{cc}
\ast & \ast \\
c & mc + d
\end{array}\right]
\]
So by splitting the inner sum into cosets of $\Gamma(N)_\infty' \backslash \Gamma(N) / \Gamma(N)_\infty'$, we get
\[
E_\infty(p,s) = t^s + \frac{1}{4}\sum_{c \in (N) - \{0\}} \frac{t^s}{\abs{c}^{2s}}\sum_{\substack{d\in 1 + (N),\\ (c,d) = (1),\\ \mod c}} \sum_{m\in (N)}\frac{1}{(\abs{z+d/c+m}^2+t^2)^s}.
\]
We are after the zeroth Fourier coefficient, so we need to integrate this over a fundamental domain $\calD$ for $\Gamma(N)_\infty'$ acting on $\CC$. We get, writing $z=x+iy$ and $\abs{\calD}$ for the area of $\calD$,
\begin{align*}
\frac{1}{\abs{\calD}}\int_{\calD}E_\infty(z,t,s)\;\mathrm{d}x\;\mathrm{d}y & = t^s + \frac{1}{4}\sum_{c\in (N)-\{0\}} \frac{t^s}{\abs{c}^{2s}} \sum_{\substack{d\in 1 + (N),\\ (c,d) = (1),\\ \mod c}} \frac{1}{\abs{\calD}}\int_{\RR^2} \frac{1}{(x^2+y^2+t^2)^s}\;\mathrm{d}x\;\mathrm{d}y \\
& = t^s + \frac{1}{4}\sum_{c\in (N)-\{0\}} \frac{t^s}{\abs{c}^{2s}} \sum_{\substack{d\in 1 + (N),\\ (c,d) = (1),\\ \mod c}} \frac{1}{\abs{\calD}}\int_0^\infty \int_0^{2\pi} \frac{r}{(r^2+t^2)^s}\;\mathrm{d}\theta\;\mathrm{d}r  \\
& = t^s + \frac{\pi}{4 \abs{\calD} \cdot (s-1)} \sum_{c\in (N)-\{0\}} \sum_{\substack{d\in 1 + (N),\\ (c,d) = (1),\\ \mod c}} \frac{1}{\abs{c}^{2s}}
\end{align*}
Let us now first compute
\begin{align*}
 \sum_{c\in (N)-\{0\}} \sum_{\substack{d\in 1 + (N),\\ (c,d) = (1),\\ \mod c}} \frac{1}{\abs{c}^{2s}}
& =
  \frac{1}{N^{2s}}\sum_{c\in \ZZ[i]-\{0\}}  \frac{1}{\abs{c}^{2s}} \card{\st{d\in \Big(1 + (N)\Big)/ (c\cdot N) }{(c,d)=1}} \\
  & =
  \frac{1}{N^{2s}}\sum_{c\in \ZZ[i]-\{0\}}  \frac{1}{\abs{c}^{2s}}  \frac{\phi(cN)}{\phi(N)},
  \end{align*}
  where $\phi$ denotes the Euler totient function of $\ZZ[i]$. The reason for this last equality is that the set $\st{d\in \Big(1 + (N)\Big)/ (c\cdot N) }{(c,d)=1}$ is exactly the kernel of the (surjective) reduction map 
\[
\Big(\ZZ[i]/(cN)\Big)^* \longrightarrow \Big(\ZZ[i]/(N)\Big)^*
\]
between the two groups of invertible elements.

For $\ZZ[i]$, the Euler product formula reads
\[
\phi(m) = \abs{m}^2 \prod_{p\in\calP \text{ s.t. }p|m} \left(1-\frac{1}{\abs{p}^2}\right), 
\]
where $\calP$ denotes the set of primes of $\ZZ[i]$. So, if we write $N=\prod_{p\in \calP} p^{a_p}$ we get
\begin{align*}
\sum_{c\in (N)-\{0\}} \sum_{\substack{d\in 1 + (N),\\ (c,d) = (1),\\ \mod c}} \frac{1}{\abs{c}^{2s}} 
& = \frac{1}{N^{2s}\phi(N)}\prod_{\calP} \sum_{k=0}^\infty \frac{\phi(p^{k+a_p})}{\abs{p}^{k\cdot 2s}} \\
& = \frac{1}{N^{2s}\phi(N)}\prod_{p\in\calP \text{ s.t. } p \nmid N}\left(1 + \left(1-\frac{1}{\abs{p}^2}\right) \sum_{k=1}^\infty \frac{\abs{p}^{2k}}{\abs{p}^{k\cdot 2s}}\right)  \\
& \quad \cdot \prod_{p\in\calP \text{ s.t. } p|N}  \left(1-\frac{1}{\abs{p}^2}\right) \sum_{k=0}^\infty \frac{\abs{p}^{2k+2a_p}}{\abs{p}^{k\cdot 2s}}  \\
& = \frac{1}{N^{2s}\phi(N)} \prod_{p\in\calP \text{ s.t. } p\nmid N} \left(1 + \left(1-\frac{1}{\abs{p}^2}\right) \cdot \frac{\abs{p}^{2-2s}}{1-\abs{p}^{2-2s}} \right) \\
& \quad \cdot \prod_{p\in\calP \text{ s.t. }p|N} \left(1-\frac{1}{\abs{p}^2}\right)\cdot \frac{\abs{p}^{2 a_p}}{1-\abs{p}^{2-2s}} \\
& = \frac{1}{N^{2s}} \prod_{p\in\calP \text{ s.t. } p\nmid N} \frac{1-\abs{p}^{-2s}}{1-\abs{p}^{2-2s}} \cdot \prod_{p\in\calP \text{ s.t.} p|N} \frac{1}{1-\abs{p}^ {2-2s}} \\
& =  \frac{\zeta_{\QQ(i)}(s-1)}{\zeta_{\QQ(i)}(s)} \cdot \frac{1}{N^{2s}} \cdot \prod_{p|N} \frac{1}{1-\abs{p}^{-2s}}
\end{align*}
Finally, we observe that $\abs{\calD}=N^2$. So, combining all of the above yields the lemma.
\end{proof}

Now we're ready to prove that $\Gamma(N)$ has no residual spectrum.

\begin{proof}[Proof of Proposition \ref{prp_residual_spectrum}] It's well known (see for instance \cite[Proposition 10.5.5]{Cohen}) that 
\[
\zeta_{\QQ(i)}(s) = \zeta(s) \cdot L(s,\chi_{-4}),
\]
where $\zeta$ denotes the Riemann zeta function and  $L(\cdot,\chi_{-4})$ is the Dirichlet $L$-function for the Legendre--Kronecker character $\chi_{-4}(n) = \left(\frac{-4}{n}\right)$. The latter is an entire function (see for instance \cite[Corollary 10.2.3(3)]{Cohen}). As such, the only pole of $\zeta_{\QQ(i)}(s-1)/\zeta_{\QQ(i)}(s)$ in the half plane $\mathrm{Re}(s)>1$ is the simple pole at $s=2$ coming from the Riemann zeta function.
\end{proof}

\subsection{Bass notes of various polyopes}

We start with a proposition on the spectrum of three polytopes in $\HH^3$ of finite volume. Recall that $\calC$ denotes the fundamental corner of the ideal right-angled regular octahedron $O$. Moreover, if $P\subset\HH^3$ is a polytope, we will write $\sigma^{\mathrm{N}}(P)$ (and $\sigma^{\mathrm{D}}(P)$ respectively) for the spectrum of the Laplacian on $L^2$-functions on $P$, subject to Neumann conditions (and Dirichlet conditions respectively) on all faces of $P$. We need the following facts:

\begin{prp}\label{prp_spectrum_building_blocks}
\begin{itemize}
\item[(a)] We have:
\[
\sigma^{\mathrm{D}}(O) \subset \sigma^{\mathrm{N}}(O) = \{0\} \cup [1,\infty).
\]
\item[(b)] Let $P_1\subset \HH^3$ denote a polytope that is isometric to the set of copies of $\calC$ that are incident to a single face of $O$. Then
\[
\sigma^{\mathrm{D}}(P_1) \subset \sigma^{\mathrm{N}}(P_1) = \{0\} \cup [1,\infty).
\]
\item[(c)] Let $P_2\subset \HH^3$ denote a polytope that is obtained by taking the union of $P_1$ with the reflection of $P_1$ along the face of $O$ that we used to define it. Then
\[
\sigma^{\mathrm{D}}(P_2) \subset \sigma^{\mathrm{N}}(P_2) = \{0\} \cup [1,\infty).
\]
\end{itemize}
Moreover, for all of the polytopes above, the smallest embedded Neumann eigenvalue is at least $\pi^2/6$.
\end{prp}

\begin{proof}
We start with item (a). First of all we note that, since $\Gamma_{\mathrm{SA}}$ is generated by the reflections in the faces of $O$, the spectrum $\sigma^{\mathrm{N}}(O)$ is exactly the spectrum on functions that are invariant under $\Gamma_{\mathrm{SA}}$. Because $\Gamma_{\mathrm{SA}}$ is a non-uniform lattice, and because $O$ is a subset of $\HH^3$ respectively, we obtain that
\[
 \{0\} \cup [1,\infty) \subset \sigma^{\mathrm{N}}(O) \quad \text{and} \quad \sigma^{\mathrm{D}}(O)\subset [1,\infty),
\]
respectively. So, we just need the reverse to the first inclusion. 

Now $\Gamma_{\mathrm{SA}}$ contains the congruence group $\Gamma(2)<\PGL(2,\ZZ[i])$ as a subgroup of index $2$ (see \cite[p. 2301]{CFH}), which in particular implies $\Gamma_{\mathrm{SA}}$ has no residual spectrum, by Proposition \ref{prp_residual_spectrum} and the fact that residual spectrum of $\Gamma_{\mathrm{SA}}$ would lift to residual spectrum of $\Gamma(2)$. That is, the spectrum in the interval $(0,1)$, if any exists, comes from cusp forms.

In order to prove a lower bound on the spectrum coming from cusp forms, we will use an idea originally due to Roelcke \cite{Roelcke} (see also \cite{ElstrodtGrunewaldMennicke_someremarks,Huxley}).  Let $f:O\to\RR$ be a cusp form such that $\norm{f}_2^2=1$ and $\norm{\nabla f}^2 = \lambda$. In other words
\[
\lambda = \int_O \frac{1}{t} \left( \abs{\frac{\partial f}{\partial x}}^2 + \abs{\frac{\partial f}{\partial y}}^2 + \abs{\frac{\partial f}{\partial t}}^2\right)\;\mathrm{d}x\;\mathrm{d}y\;\mathrm{d}t.
\]
Now set
\[
B = \st{(x,y,t)\in \HH^3}{0\leq x,y \leq 1,\; t\geq \frac{1}{\sqrt{2}}}.
\]
Recall that $G$ denotes the group of isometries of $O$. We observe that $6$ elements of $G$ suffice to cover $O$ with translates of $B$ (one for every cusp of $O$). That is, there exist $g_1,\ldots,g_6\in G$ such that
\[
O = \bigcup_{i=1}^6 g_i B.
\]
Moreover, we claim that every point of $O$ except the point $c=\left(\frac{1}{2},\frac{1}{2},\frac{1}{\sqrt{2}}\right)$ is covered by at most $3$ of these translates of $B$. To see this, we note that $B$ is the intersection of a horoball with $O$. As such, all the $g_i B$ are intersections of horoballs with $O$. Because the point $c$ is a global fixed point of $G$, all the translates $g_i B$ pass through this point as well. So every $g_i B$ passes through one of the cusps and the point $c$. This already determines the translate incident to the cusp at $(i+1)/2$, which intersects $B$ only at the point $c$. For the other translates we need one more datum to determine them. We use the fact that $B$ intersects the boundary of $O$ orthogonally and hence so do its translates. These translates are thus euclidean balls that intersect the square 
\[
\st{(x,y,z)\in B}{ t = \frac{1}{\sqrt{2}}}
\]
in a quarter disk of radius $\frac{1}{\sqrt{2}}$. Figure \ref{pic_horoballs} shows the horoballs that intersect $B$ (in $B$ on the left and in the square on the right).
\begin{figure}
\begin{center}
\includegraphics[scale=1]{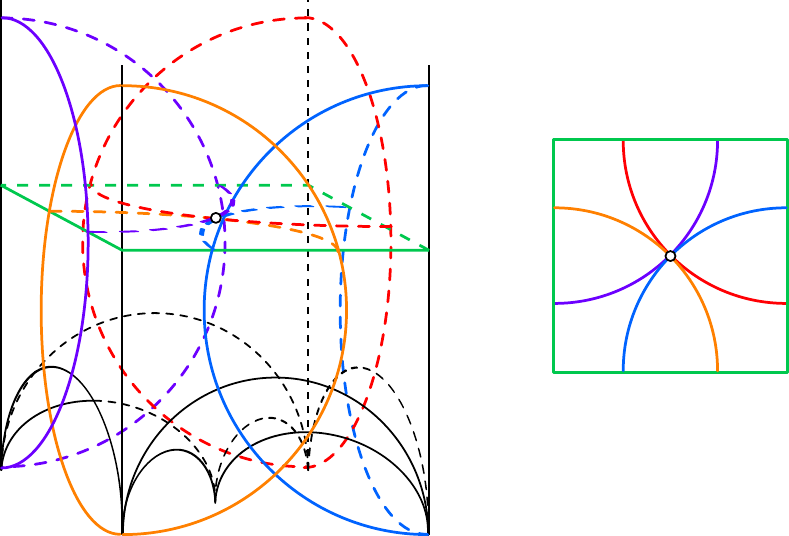}
\caption{The horoballs intersecting $B$.}\label{pic_horoballs}
\end{center}
\end{figure}
In particular, the number of translates that contain a given point of $O-\{c\}$ is at most $3$.

So we get that 
\[
3 \cdot \lambda > \sum_{i=1}^6 \int_B \frac{1}{t}\left( \abs{\frac{\partial f\circ g_i^{-1}}{\partial x}}^2 +  \abs{\frac{\partial f\circ g_i^{-1}}{\partial y}}^2 +
 \abs{\frac{\partial f\circ g_i^{-1}}{\partial t}}^2 \right) \; \mathrm{d}x\;\mathrm{d}y\;\mathrm{d}t.
\]
Now we will estimate this sum term by term. The function $f$ is invariant under the parabolic subgroups corresponding to the cusps. That is, $f\circ g_i^{-1}$ is invariant under $z\mapsto z+2$, $z\mapsto z+2i$ and $x+iy \mapsto \pm x \pm iy$. Moreover, $f$ is a cusp form, so it admits a Fourier series:
\[
f\circ g_i^{-1}(x,y,t) = \sum_{(m,n)\in\NN^2-\{(0,0)\}} a_{m,n}^{(i)}(t) \cos(\pi mx) \cos(\pi ny).
\]
In order to integrate over a full period of our cosine functions, we write
\[
C = B\cup (B+1) \cup (B+i)\cup (B+1+i)
\]
and obtain:
\begin{multline*}
4\cdot \int_B \frac{1}{t}\left( \abs{\frac{\partial f\circ g_i^{-1}}{\partial x}}^2 +  \abs{\frac{\partial f\circ g_i^{-1}}{\partial y}}^2 +
 \abs{\frac{\partial f\circ g_i^{-1}}{\partial t}}^2 \right) \; \mathrm{d}x\;\mathrm{d}y\;\mathrm{d}t \\
 \geq  \int_C \frac{1}{t}   \abs{\sum_{(m,n)\in\NN^2-\{(0,0)\}} -\pi m\cdot a_{m,n}^{(i)}(t) \sin(\pi m x)\cos(\pi n y)}^2  \mathrm{d}x\;\mathrm{d}y\;\mathrm{d}t \\
 +  \int_C \frac{1}{t}   \abs{\sum_{(m,n)\in\NN^2-\{(0,0)\}} -\pi n\cdot a_{m,n}^{(i)}(t) \cos(\pi m x)\sin(\pi n y)}^2  \mathrm{d}x\;\mathrm{d}y\;\mathrm{d}t.
\end{multline*}
Applying Parseval's identity yields
\begin{multline*}
4\cdot \int_B \frac{1}{t}\left( \abs{\frac{\partial f\circ g_i^{-1}}{\partial x}}^2 +  \abs{\frac{\partial f\circ g_i^{-1}}{\partial y}}^2 +
 \abs{\frac{\partial f\circ g_i^{-1}}{\partial t}}^2 \right) \; \mathrm{d}x\;\mathrm{d}y\;\mathrm{d}t \\
   \geq 4  \sum_{(m,n)\in\NN^2-\{(0,0)\}} \pi^2 (m^2+n^2)\int_{1/\sqrt{2}}^\infty \frac{1}{t} \abs{a_{m,n}^{(i)}(t)}^2 \;\mathrm{d}t  \\
  > 4\pi^2  \sum_{(m,n)\in\NN^2-\{(0,0)\}} \int_{1/\sqrt{2}}^\infty \frac{1}{t} \abs{a_{m,n}^{(i)}(t)}^2 \;\mathrm{d}t \\
  > 2\pi^2 \sum_{(m,n)\in\NN^2-\{(0,0)\}} \int_{1/\sqrt{2}}^\infty \frac{1}{t^3} \abs{a_{m,n}^{(i)}(t)}^2 \;\mathrm{d}t  \\
= \frac{\pi^2}{2} \int_C \frac{1}{t^3} \abs{f\circ g_i^{-1}}^2 \; \mathrm{d}x\;\mathrm{d}y\;\mathrm{d}t 
 = 2\pi^2 \int_B \frac{1}{t^3} \abs{f\circ g_i^{-1}}^2 \; \mathrm{d}x\;\mathrm{d}y\;\mathrm{d}t ,
\end{multline*}
where we've applied Parseval again in the penultimate step and the reflection symmetry of $f$ in the last. All in all, we obtain

\begin{align*}
12\cdot \lambda  & > 2\pi^2 \sum_{i=1}^6 \int_B \frac{1}{t^3} \abs{f\circ g_i^{-1}}^2 \mathrm{d}x\;\mathrm{d}y\;\mathrm{d}t \\
 & > 2\pi^2 \cdot \norm{f}_2^2
\end{align*}
and hence $\lambda > \pi^2/6$.

For items (b) and (c) we observe that both $P_1$ and $P_2$, thought of as orbifolds can be obtained as the quotient orbifold of $O$ by a finite group. This means that their spectrum lifts to $O$. 
\end{proof}

We will also derive some crude estimates on the number $\lambda_0(T)$. It would be very interesting to have more precise numbers than what we obtain below.

Before we prove our bounds, we recall the definition of the \textbf{critical exponent} of a countable group $\Gamma$ acting properly on a space $X$.
\begin{equation}\label{eq_def_crit_exp}
\delta(\Gamma,X) = \inf\st{s\in(0,\infty)}{ \sum_{\gamma \in \Gamma} e^{-s\cdot \dist_X(x_0,\gamma \cdot x_0)} \text{ converges}}.
\end{equation}
If $X=\HH^3$ then we will often drop it from the notation and just write $\delta(\Gamma)=\delta(\Gamma,\HH^3)$. As the notation suggests, this number does not depend on the choice of base point $x_0$.  It's moreover related to the bass note of $\Gamma\backslash\HH^3$ through the Elstrodt--Sullivan formula\footnote{We note that Sullivan states this formula for torsion-free groups. All our groups are virtually torsion free -- even the non-finitely generated groups we consider -- so the theory still applies.} \cite[Theorem 2.17]{Sullivan}:
\begin{equation}\label{eq_ElstrodtSullivan}
\lambda_0(\Gamma) = \left\{\begin{array}{ll}
1 & \text{if } \delta(\Gamma) \leq 1 \\
\delta(\Gamma)\cdot (2-\delta(\Gamma)) & \text{if } \delta(\Gamma) > 1.
\end{array}\right.
\end{equation}

Combining this with recent work by Coulon \cite{Coulon}, we will derive an upper bound on $\lambda_0(T)$. Our lower bound will come from the Brooks--Burger transfer principle \cite{Brooks_transfer,Burger}.

\begin{prprep}{\ref*{prp_main}}
We have
\[
\frac{3-\sqrt{5+2\sqrt{3}}}{67-\sqrt{5+2\sqrt{3}}} \leq \lambda_0(T) \leq  \delta(\Gamma_{\mathrm{Ap}})  \cdot (1- \delta(\Gamma_{\mathrm{Ap}})  /4)
\]
\end{prprep}

\begin{proof}
To be on the safe side, we note that $\Gamma_\infty$ is indeed virtually torsion free. It's abstractly isomorphic to a free product of countably many copies of $\ZZ/2\ZZ$:
\[
\Gamma_\infty \simeq \langle t_1,t_2,\ldots | t_n^2 = 1\; \forall n\in \NN\rangle.
\]
This group admits an index two torsion-free subgroup: the kernel of the map $\Gamma_\infty\to\ZZ/2\ZZ$ sending each generator $t_i$ to the generator of $\ZZ/2\ZZ$. Geometrically this index two subgroup is the fundamental group of the double of the octa-tree along its boundary. Since neither the critical exponent, nor $\lambda_0$ change when passing to a finite index subgroup, \eqref{eq_ElstrodtSullivan} applies.

Now for the upper bound we observe that $\Gamma_\infty$ contains a copy of the Apollonian group, which implies that its critical exponent $\delta(\Gamma_\infty)$ satisfies 
\[
\delta(\Gamma_\infty) \geq \delta(\Gamma_{\mathrm{Ap}}).
\]
We have just noted that the latter is larger than $1$. So \eqref{eq_ElstrodtSullivan} yields an upper bound on $\lambda_0(T)$ (of $0.906555\ldots$). We can however do slightly better using \cite{Coulon}. Coulon proved that (when specialized to our setting)
\[
\delta\Big(\Gamma_\infty, \HH^3\Big) \geq \delta\Big(\Gamma_{\mathrm{SA}}, \HH^3\Big) - \frac{1}{2} \delta\Big( \Gamma_{\mathrm{SA}} / \Gamma_\infty, \Gamma_\infty\backslash\HH^3\Big).
\]
The orbifold $\Gamma_\infty\backslash\HH^3$ can be identified with the octa-tree $T$ (with the boundary interpreted as mirrors). This means that $\delta\Big( \Gamma_{\mathrm{SA}} / \Gamma_\infty, \Gamma_\infty\backslash\HH^3\Big) = \delta(\Gamma_{\mathrm{Ap}})$. Moreover, $\Gamma_{\mathrm{SA}}$ is a lattice in $\Isom(\HH^3)$, so $\delta(\Gamma_{\mathrm{SA}})=2$. Filling this in gives the bound we claimed.

For the lower bound, we use the effective version of the Brooks--Burger transfer method (see \cite{Brooks_transfer,Burger}) as written up by Breuillard \cite{Breuillard} (see \cite{Tapie} for an alternative version in dimension two). Technically Breuillard only works with finite degree covers, but the argument goes through in our case. The only difference is that we cannot assume that there is an eigenfunction for $\lambda_0(T)$, which means that in the argument, the eigenfunction needs to be replaced by (for instance) a $C^1$ function $f:T\to\RR$ of compact support such that
\[
\frac{\norm{\nabla f}_2^2}{\norm{f}_2^2} \leq \lambda_0(T) + \eps
\]
for some $\eps>0$ that we will send to $0$ a posteriori (as Tapie does in his paper). The result in our case is that if $\Gamma<\mathrm{Isom}(\HH^3)$ is a lattice and $\Lambda<\Gamma$ is a subgroup such that $\Lambda \backslash\HH^3 = T$, then
\[
\lambda_0(T) \geq \frac{\mu\cdot \lambda_1(\calG)}{16\cdot d_{\calG}+\mu\cdot\lambda_1(\calG)},
\]
where $\calG$ denotes the Schreier graph of $\Gamma$ acting on $\Gamma/\Lambda$, $d_{\calG}$ denotes the degree of this graph and 
\[
\mu = \min_{(v,w) \text{ edges of }\calG}\{\lambda_1^{\mathrm{N}}(\calF(v)),\lambda_1^{\mathrm{N}}(\calF(v)\cup\calF(w)\},
\]
where $\lambda_1^{\mathrm{N}}$ denotes the bottom of the Laplace spectrum with Neumann boundary conditions. Moreover, we think of $\calG$ as embedded in $T$, dual to a decomposition of $T$ into isometric copies of a fundamental domains for $\Gamma \backslash\HH^3$ and hence we can talk about the fundamental domain $\calF(v)$ associated to a vertex $v$ of $\calG$. Note that the fact that $\calF(v_1)$ and $\calF(v_2)$ are isometric for any vertices $v_1$ and $v_2$ implies that the minimum above is the minimum of a finite number of strictly positive bass notes.

We will use the polytope $P_1$ from Proposition \ref{prp_spectrum_building_blocks}. That is, $\Gamma$ will denote the group generated by the reflections in the faces of $P_1$. This group contains $\Gamma_{\mathrm{SA}}$ as an index $4$ subgroup. As such it also contains $\Gamma_\infty$. Because the dual graph to the decomposition of $O$ into copies of $P_1$ is a complete graph on $4$ vertices, the Schreier graph $\calG$ of $\Gamma$ acting on $\Gamma_\infty$ is the $4$-regular graph obtained from the $4$-regular tree by replacing every vertex with a complete graph $K_4$ on $4$ vertices in such a way that every vertex of every copy of $K_4$ connects to exactly one edge coming from the original tree (Figure \ref{pic_dualgraph} shows a local picture). In graph theory, this is called a replacement product. 
\begin{figure}
\begin{center}
\includegraphics[scale=1]{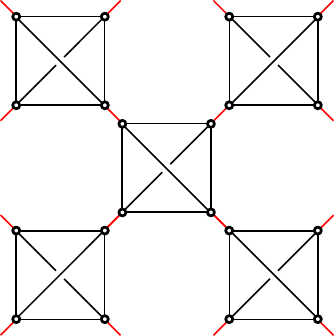}
\caption{The local structure of the dual graph $\calG$ to the tiling of $T$ by copies of $P_1$.}\label{pic_dualgraph}
\end{center}
\end{figure}

We can also think of $\calG$ as a Cayley graph for $(\ZZ/4\ZZ)\ast (\ZZ/2\ZZ)$ with respect to the generating set that contains all the non-trivial elements of the $(\ZZ/4\ZZ)$ factor and the non-trivial element of the $(\ZZ/2\ZZ)$ factor. As such, it follows from \cite[Theorem 1]{CartwrightSoardi} that
\[
\lambda_1(\calG) = 3-\sqrt{5+2\sqrt{3}}
\]

So, using Proposition \ref{prp_spectrum_building_blocks}, and the fact that the polytopes $\calF(v)\cup\calF(w)$ that appear in the definition of $\mu$ are copies of $P_1$ and $P_2$, we get:
\[
\lambda_0(T) \geq \frac{3-\sqrt{5+2\sqrt{3}}}{67-\sqrt{5+2\sqrt{3}}}.
\]
\end{proof}

\subsection{The bass note of the level $2$ congruence subgroup}

Using the results from the previous section, we also obtain that the spectrum of $\Gamma(2)$ starts at $1$:
\begin{prp}\label{prp_eigenvalue_gamma_2}
We have
\[
\lambda_1(\Gamma(2)\backslash \HH^3)\geq \lambda_1^{\mathrm{D}}(O) > 1.
\]
\end{prp}

\begin{proof}
As we mentioned in Section \ref{sec_SA}, $\Gamma_{\mathrm{SA}}=\Gamma(2)\rtimes \langle \rho\rangle$. In particular, $\Gamma(2)$ coincides with index $2$ the subgroup of orientation preserving elements of $\Gamma_{\mathrm{SA}}$. As such, the spectrum of the Laplacian on $\Gamma(2)\backslash\HH^3$ on even (respectively odd) functions with respect to the Deck group action of $\ZZ/2\ZZ$ coincides can be thought of as the spectrum of the Laplacian on functions on $O$ with Neumann (respectively Dirichlet) boundary conditions. In the proof of Proposition \ref{prp_spectrum_building_blocks}, we proved that neither contains eigenvalues $\leq 1$.
\end{proof}

\section{Near optimal spectral gaps}\label{sec_near_optimal}

Given the results from the previous section, we can now combine results on strong convergence of permutation representations due to Bordenave--Collins \cite{BordenaveCollins} with techniques due to Hide--Magee, Moy and Ballman--Mondal--Polymerakis \cite{HideMagee,Moy,BallmanMondalPolymerakis} to prove that both our models $S_n$ and $M_n$ of random covers have near optimal spectral gaps. In this section we will explain how this works. This will yield Theorems \ref{thm_main1} and \ref{thm_main2}.

For any geometrically finite hyperbolic $d$-orbifold $M$, the spectrum of the Laplacian $\Delta_{M}$ below $\left(\frac{d-1}{2}\right)^2$ is discrete, consisting of finitely many eigenvalues of finite multiplicity \cite{LaxPhillips}.  We say an eigenvalue in a cover of $M$ is $\textbf{new}$ if it does not arise by lifting from $M$. For a complete hyperbolic manifold $X$ we write $\lambda_0(X)$ to denote the bottom of the $L^2$-spectrum of the Laplacian on $X$.

Moreover, we will write
\[
\ell_0^2(\{1,\ldots,n\}) = \st{v\in\ell^2(\{1,\ldots,n\})}{\sum_{i=1}^n v_i=0}.
\]
The symmetric group $\sym_n$ acts on this vector space by permuting the coordinates. This representation $\mathrm{std}_n:\sym_n \to U(\ell_0^2(\{1,\ldots,n\}))$ is irreducible and is called the \textbf{standard representation} of $\sym_n$. Here and below $U(\calH)$ denotes the group of unitary operators on a Hilbert space $\calH$. Finally, if $F$ is a finitely generated group, we will write $\lambda_F:F\to U(\ell^2(F))$ for the \text{left regular representation} of $F$, defined by
\[
\Big(\lambda_F(g)\cdot f\Big)(h) = f(g^{-1}h),\quad g,h\in F,\; f\in\ell^2(F).
\]

We now have the following theorem:

\begin{thm}
Let $M=\Gamma \backslash \mathbb{H}^d$ be a $d$-dimensional geometrically finite hyperbolic orbifold and let $\phi:\Gamma \to F$ be a surjection, where $F$ is a finitely generated group. Let $\{\varphi_i\in \Hom(F,\sym_{n_i})\}_{i\in \mathbb{N}}$ be so that the permutation representations $\left(\textup{std}_{n_i}\circ \varphi_i,\ell^2_0\left(\{1,\dots,n_i\}\right)\right)$ of $F$ strongly converge to $\left(\lambda_F,\ell^2(F)\right)$. Then the associated covers $M_i=\textup{Stab}_{\varphi_i\circ \phi}(1)\backslash \mathbb{H}^d$ have no new eigenvalues below $$\lambda_0\left(\textup{ker}(\phi)\backslash \mathbb{H}^d\right)-o_{i\to \infty}(1).$$
\end{thm}

This can be obtained from the arguments of \cite{HideMagee} with only some minor adaptations to deal with the model resolvent in the ends. The arguments of \cite{Moy} and also \cite{BallmanMondalPolymerakis} extend immediately to  give the conclusion.

We will combine this with a special case of \cite[Theorem 3]{BordenaveCollins}:
\begin{thm}[Bordenave--Collins]
Let $F=(\ZZ/2\ZZ)^{\ast 4}$ and $\varphi_n:F \to \sym_{2n}$ be as in Section \ref{sec_model} above, then in probabiltiy, $\mathrm{std}_{2n}\circ\varphi_n:F\to U(\ell_0^2(\{1,\ldots,2n\}))$ strongly converges to $(\lambda_F,\ell^2(F))$ as $n\to\infty$. 
\end{thm}

This, together with the work we have done in Section \ref{sec_spectral_input} allows us to prove Theorems \ref{thm_main1} and \ref{thm_main2}.

\begin{proof}[Proof of Theorems \ref*{thm_main1} and \ref*{thm_main2}] 
Theorem \ref*{thm_main1} follows from the two theorems above, together with the fact that there is only one old eigenvalue coming from $\Gamma_{\mathrm{Ap}}\backslash\HH^3$, as proved in \cite{KelmerKontorovichLutsko}.

Let us now prove Theorem \ref{thm_main2}. We will write $M_n = \Gamma_n\backslash\HH^3$. Note that this is a slight abuse of notation, because it happens that $M_n$ is disconnected. However, the asymptotic probability of this event is $0$. Theorem \ref{thm_main2} (1) is now a consequence of the two theorems above, combined with Proposition \ref{prp_spectrum_building_blocks}, which implies that there is no old spectrum (spectrum of $\Gamma_{\mathrm{SA}}\backslash\HH^3$) in the interval $(0,1)$. 

To prove Theorem \ref{thm_main2} (2), we note that the spectrum on even functions on $DM_n$ comes from the spectrum on $M_n$, so this has been settled. The spectral gap on odd functions converges to the spectral gap on odd functions on the orientation double cover $D(\Gamma_\infty\backslash\HH^3)$ of $\Gamma_\infty\backslash\HH^3$ (again using the fact that by Proposition \ref{prp_spectrum_building_blocks} there is no old spectrum). This is the double of the octa-tree. So this is the same as the bottom of the Dirichlet spectrum of the octa-tree. The octa-tree is a closed domain in $\HH^3$. So, by domain monotonicity, the bottom of its Dirichlet spectrum lies above that of $\HH^3$, which means it's above $1$.
\end{proof}

\section{Delocalization}

Now that we have proved Theorems \ref{thm_main1} and \ref{thm_main2}, we start preparing for the proof of Theorem \ref{thm_main3}. In this section we prove delocalization estimates on eigenfunctions. First we prove a delocalization result that is a $3$-dimensional version of \cite[Theorem 4.1]{GLMST} and \cite[Proposition 6.1]{Magee_bassnotes}. After that, we introduce a variant of tangle-freeness for manifolds that locally look like a cover (not necessarily their universal cover) and prove a bound on the $L^\infty$-norm of eigenfunctions on manifolds that have this property.

\subsection{Spherical caps}
Before we get to it, we need the volume of a spherical cap in hyperbolic $3$-space:
\begin{lem}\label{lem_spherical_cap_volume}
Let $x,y\in\HH^3$ and set $\delta = \dist(x,y)$, then
\begin{multline*}
\vol\Big(\st{z\in \HH^3}{\dist(z,x)\leq T \text{ and } d(z,y)\leq T}\Big) \\ = 4\pi \int_0^{r_0} \sinh^2(r) \mathrm{d}r + 2\pi \int_{r_0}^T \sinh^2(r) \cdot \left(1 - \frac{\cosh(r)\cosh(\delta)-\cosh(T)}{\sinh(r)\sinh(\delta)}\right) \mathrm{d}r,
\end{multline*}
where $r_0 = \max\{0,T-\delta\}$.
\end{lem}

\begin{proof}
We use polar coordinates $z=(r,\theta,\varphi)$ in which $r$ denotes the hyperbolic distance to $x$ and $\theta$ and $\phi$ are measured with respect to the geodesic $\alpha$ through $x$ and $y$: $\theta$ is the angle between $\alpha$ and the geodesic through $z$ and $x$ and $\varphi$ is the angle around the axis $\alpha$. With respect to these coordinates, the volume element takes the form
\[
\mathrm{d}\vol_{\HH^3} = \sinh^2(r) \; dA_{\Sphere^2} \;\mathrm{d}r,
\]
where $dA_{\Sphere^2} = \sin(\theta)d\theta d\varphi$ denotes the area element on the $2$-sphere of radius $1$. 

Using for instance the formulas of \cite[p. 454]{Buser_book}, we deduce that
\[
\cosh\Big(\dist\Big( (r,\theta,\varphi),y\Big)\Big) = \cosh(r)\cosh(\delta)-\sinh(r)\sinh(\delta) \cos(\theta). 
\]
This means that in these coordinates
\begin{align*}
C & = \st{z\in \HH^3}{\dist(z,x)\leq T \text{ and } d(z,y)\leq T} \\
& = \st{(r,\theta,\varphi)}{ r\leq r_0 } \\
& \quad \quad \cup \st{(r,\theta,\varphi)}{ r_0 \leq r \leq T \text{ and } \cos(\theta) \geq \frac{\cosh(r)\cosh(\delta)-\cosh(T)}{\sinh(r)\sinh(\delta)} },
\end{align*}
where $r_0 = \max\{0,T-\delta\}$.

Now the area of a spherical cap in $\Sphere^2$, defined by all points of angle at most $\theta$, is $2\pi(1-\cos(\theta))$. This means that all in all,
\[
\vol(C)  = 4\pi \int_0^{r_0} \sinh^2(r) \mathrm{d}r + 2\pi \int_{r_0}^T \sinh^2(r) \cdot \left(1 - \frac{\cosh(r)\cosh(\delta)-\cosh(T)}{\sinh(r)\sinh(\delta)}\right) \mathrm{d}r.
\]
\end{proof}

We note that the integrals above can be evaluated and this leads to a closed formula. However, the length of this formula is such that it's hard to extract more information from it than from the integral expression. We will bound the volume as follows:
\begin{lem}\label{lem_spherical_cap_bound}
There exists a universal constant $C>0$ such that for all $x,y\in\HH^3$ 
\[
\vol\Big(\st{z\in \HH^3}{\dist(z,x)\leq T \text{ and } d(z,y)\leq T}\Big) \leq  C\cdot e^{2T-\dist_{\HH^3}(x,y)}.
\]
\end{lem}
\begin{proof}
We have $r_0\leq T-\dist_{\HH^3}(x,y)$, which means that
\[
4 \pi \int_0^{r_0} \sinh^2(r)\mathrm{d}t \ll e^{2T-2\dist_{\HH^3}(x,y)}. 
\]
Likewise
\[
2\pi \int_{r_0}^T \frac{\sinh(r)\cosh(T)}{\sinh(\dist_{\HH^3}(x,y))} \mathrm{d}r \ll e^{2T-\dist_{\HH^3}(x,y)}.
\]
Finally, using that $\cosh(x)\geq \sinh(x)$,
\[
2\pi\int_{r_0}^T \sinh^2(r) \left(1-\frac{\cosh(r)\cosh(\dist_{\HH^3}(x,y))}{\sinh(r)\sinh(\dist_{\HH^3}(x,y))}\right)\mathrm{d}r \leq 0
\]
\end{proof}

\subsection{Bounds coming from the pre-trace formula}
We now state the delocalization result we need:

\begin{prp}\label{prp_deloc}
Let $\Gamma<\Isom(\HH^3)$ be a geometrically finite discrete group. Moreover suppose that $\lambda < 1$ and that $f:\Gamma\backslash\HH^3\to\RR$ is an $L^2$-normalized Laplacian eigenfunction of eigenvalue $\lambda$. Then
\[
\abs{f(x)}^2 \ll \frac{1-\lambda}{\sinh^2(T\sqrt{1-\lambda})} \cdot   \sum_{\substack{\gamma\in\Gamma \text{ s.t.}\\ \dist_{\HH^3}(x,\gamma x) \leq T}}  e^{-\dist_{\HH^3}(x,\gamma x) } ,
\]
for all $x\in\Gamma\backslash\HH^3$ and all $T\geq 1$. The implied constant in the above is independent of both $T$ and $\lambda$.
\end{prp}

\begin{proof} First, we need a version of \cite[Proposition 5.2]{Gamburd}. Let $K:[0,\infty)\to\RR$ be compactly supported such that its Selberg transform $H$ is non-negative. Then for any $x\in\HH^3$,
\[
\sum_{j:\;\lambda_j \leq 1}  H(\lambda_j) \abs{f_j(x)}^2 \leq \sum_{\gamma\in\Gamma} K\Big(\dist_{\HH^3}(\gamma\cdot x,x)\Big),
\]
where $0=\lambda_0 < \lambda_1 \leq \ldots \leq \lambda_q \leq 1$ are the small Laplacian eigenvalues of $\Gamma\backslash\HH^3$ and $f_1,\ldots f_q$ are and $L^2$-normalized basis of the corresponding eigenspaces (the fact that this is a finite set follows from \cite{LaxPhillips}). The proof of this bound is the same as Gamburd's proof, applied in dimension $3$ (see also \cite[(5.2)]{Magee_quantitative}).

The Selberg transform for $\HH^3$ is given by
\[
H(\lambda) = \frac{2\pi}{s}\int_0^\infty K(r) \sinh(s r)\sinh(r) \mathrm{d}r
\]
for $\lambda = 1-s^2$ (see for instance \cite[Theorem 3.5.3]{ElstrodtGrunewaldMennicke}). This is a $\ast$-morphism with respect to convolution of functions, which we will use to get the non-negativity we need.

We will fix some parameter $T\geq 1$ and set 
\[
k_T(r) = \frac{\ind_{[0,T]}(r)}{\sinh(T)}
\]
so that for $\lambda\in (0,1)$
\[
h_T(\lambda) = \frac{2\pi}{\sinh(T)\cdot s}\int_0^T \sinh(s r)\sinh(r)\mathrm{d}r \gg \frac{\sinh\left(T \sqrt{1-\lambda} \right)}{\sqrt{1-\lambda}},
\]
where the implied constant is independent of both $\lambda$ and $T$.

Now we set $K_T=k_T\ast k_T$, which means that its Selberg transform $H$ satisfies
\[
H_T(\lambda) = \abs{h_T(\lambda)}^2 \gg \frac{\sinh^2\left(T\cdot\sqrt{1-\lambda} \right) }{1-\lambda}.
\]
The function $K_T$ itself satisfies
\[
K(\dist_{\HH^3}(x,y)) = \frac{1}{\sinh^2(T)} \int_{\HH^3} \ind_{B_T(x)}(z) \cdot \ind_{B_T(z)}(y)\; \mathrm{d}\vol_{\HH^3}(z). 
\]
Now we use Lemma \ref{lem_spherical_cap_bound} and put all of the above together, to obtain
\begin{multline*}
 \frac{\sinh^2\left(T\cdot \sqrt{1-\lambda}\right)}{1-\lambda} \cdot \abs{f(x)}^2 \ll H_T(\lambda) \cdot \abs{f(x)}^2 \\
  \leq \sum_{\gamma\in\Gamma} K(\dist_{\HH^3}(\gamma\cdot x,x)) \ll \sum_{\gamma\in\Gamma \text{ s.t. } \dist_{\HH^3}(x,\gamma x) \leq T}  e^{-\dist_{\HH^3}(x,\gamma x)} ,
\end{multline*}
which implies the proposition.
\end{proof}

\subsection{The contribution from the normal subgroup}

Now we will prove a bound on the contribution coming from a normal subgroup to the sum on the right hand side in Proposition \ref{prp_deloc}. How close a point is to the cusps will play a role in this, so we will use the quantities that we've introduced in Section \ref{sec_geometry_of_cusps}. We note that in the lemma below, since $\Lambda$ is considered to be fixed throughout, we could absorb the areas $\area_{\Lambda}(\mathbf{a})$ and lengths $\mathrm{length}_{\Lambda}(\mathbf{a})$ in the constants. However, this dependence will play a role when we consider similar bounds for finite index subgroups of $\Gamma_0$, so we've decided to keep them in.

\begin{lem}\label{lem_kernel_growth}
Suppose that $\Gamma_0 < \Isom(\HH^3)$ is a lattice and $\Lambda \triangleleft \Gamma_0$ has critical exponent $\delta=\delta(\Lambda)$. Then for every $\eps>0$, there exists a constant $C>0$, depending on $\Lambda$, $\Gamma_0$ and $\eps$ only, such that for all $x\in\HH^3$:
\[
\sum_{\substack{\gamma\in\Lambda \text{ s.t.} \\ \dist_{\HH^3}(x,\gamma \cdot x) \leq T}} e^{-\dist_{\HH^3}(x,\gamma \cdot x)} \leq C \cdot \Big( e^{(\delta-1+\eps)\cdot  T} + R_{\mathrm{cusp}}(x,T)\Big),
\]
where
\[
R_{\mathrm{cusp}}(x,T) = \sum_{\substack{\mathbf{a} \text{ a cusp of }\Gamma_0 \\ \text{such that }t_{\mathbf{a}}(x)\geq 1}} R_{\mathbf{a}}(x,T)
\]
and
\[
R_{\mathbf{a}}(x,T) = \left\{ \begin{array}{ll}
 \frac{t_{\mathbf{a}}(x)^2}{\mathrm{area}_{\Lambda}(\mathbf{a})} \cdot \log\left( \frac{t_{\mathbf{a}}(x)^2\cdot \sinh(T/2)}{\mathrm{area}_{\Lambda}(\mathbf{a})}\right) & \text{if } \mathrm{rank}_{\Lambda}(\mathbf{a}) = 2 \\[2mm]
 \frac{t_{\mathbf{a}}(x)}{\mathrm{length}_{\Lambda}(\mathbf{a})} \cdot \log\left( \frac{t_{\mathbf{a}}(x)\cdot \sinh(T/2)}{\mathrm{length}_{\Lambda}(\mathbf{a})}\right)  & \text{if }  \mathrm{rank}_{\Lambda}(\mathbf{a}) = 1 \\[2mm]
 0 & \text{otherwise}.
\end{array}\right.
\]
\end{lem}

\begin{proof}
We will write $G = \Gamma_0/\Lambda$. This group acts on $\Lambda\backslash\HH^3$ with quotient $\Gamma_0\backslash\HH^3$. We will choose a fundamental domain $\calF$ for this action.

Now we fix some $\eps_0>0$ below the Margulis constant for $\HH^3$ and below half the systole of $\Gamma_0\backslash\HH^3$. We moreover let $K\subset \Lambda\backslash\HH^3$ denote the pre-image of the $(\geq \eps_0)$-thick part of $\Gamma_0\backslash\HH^3$. We note that $K$ is $G$-invariant and that the $G$-action on $K$ is cocompact. If $\Gamma_0$ contains no parabolic elements, then (because of the choice of $\eps_0$), $K = \Gamma_0\backslash\HH^3$. If not, the complement of its image in $\Gamma_0\backslash \HH^3$ consists of disjoint horosphere neighborhoods of the cusps of $\Gamma_0\backslash \HH^3$.

We once and for all fix some $x_0\in\HH^3$ that projects to a point in $\calF\cap K$ under the quotient map $\pi:\HH^3\to\Lambda\backslash\HH^3$. By definition of the critical exponent, for every $\eps'>0$,
\[
\card{\st{\gamma\in\Lambda}{\dist_{\HH^3}(x_0,\gamma\cdot x_0) \leq T}} \ll_{\eps'} e^{(\delta+\eps')\cdot T}
\]
and hence, by choosing our $\eps'<\eps$,
\[
\sum_{\substack{\gamma\in\Lambda \text{ s.t.} \\ \dist_{\HH^3}(x_0,\gamma \cdot x_0) \leq T}} e^{-\dist_{\HH^3}(x_0,\gamma \cdot x_0)} \ll_{\eps,x_0} e^{(\delta-1+\eps)\cdot T}.
\]
In order to turn this into a uniform bound for all $x\in\HH^3$, we will compare the value at $x$ to that at $x_0$.

First of all, we will argue that we can reduce to points in $\pi^{-1}(\calF)$. To see this, suppose that $g\pi(x) \in  \calF$ for some $g\in G$. Choose some $\eta \in \Gamma_0$ that projects to $g\in G$ under the quotient map $\Gamma_0 \to G$ (which means that $\pi(\eta x)\in\calF$). We then have, using that $\Lambda$ is normal in $\Gamma$ and hence $\eta\Lambda = \Lambda \eta$,
\[
\sum_{\substack{\gamma\in\Lambda \text{ s.t.} \\ \dist_{\HH^3}(x,\gamma \cdot x) \leq T}} e^{-\dist_{\HH^3}(x,\gamma \cdot x)} = 
\sum_{\substack{\gamma\in\Lambda \text{ s.t.} \\ \dist_{\HH^3}(\eta x,\eta\cdot \gamma \cdot x) \leq T}} e^{-\dist_{\HH^3}(\eta x,\eta\cdot \gamma \cdot x)} = \sum_{\substack{\gamma\in\Lambda \text{ s.t.} \\ \dist_{\HH^3}(\eta x,\gamma \cdot \eta  x) \leq T}} e^{-\dist_{\HH^3}(\eta x,\gamma \cdot \eta x)}
\]
so indeed, the maximum of this quantity (as a function of $x$, for fixed $T$) is realized in $\pi^{-1}(\calF)$.

Now we first deal with points in the thick part. That is, we write $K_0=K\cap\calF$, which is a compact set, and we suppose that $x\in \pi^{-1}(K_0)$. By the triangle inequality,
\begin{multline*}
\dist_{\HH^3}(x,\gamma\cdot x) \leq \dist_{\HH^3}(x,x_0) + \dist_{\HH^3}(x_0,\gamma\cdot x_0) + \dist_{\HH^3}(\gamma \cdot x_0,\gamma \cdot x) \\
 \leq  \dist_{\HH^3}(x_0,\gamma\cdot x_0) + 2\diam(K_0)
\end{multline*}
and likewise
\[
\dist_{\HH^3}(x,\gamma\cdot x) \geq \dist_{\HH^3}(x_0,\gamma\cdot x_0) - 2 \diam(K_0).
\]
Because the diameter $ \diam(K_0)$ depends on $\Lambda$ and $\Gamma_0$ only, we get that
\[
\sum_{\substack{\gamma\in\Lambda \text{ s.t.} \\ \dist_{\HH^3}(x,\gamma \cdot x) \leq T}} e^{-\dist_{\HH^3}(x,\gamma \cdot x)} \ll 
\sum_{\substack{\gamma\in\Lambda \text{ s.t.} \\ \dist_{\HH^3}(x_0,\gamma \cdot x_0) \leq T}} e^{-\dist_{\HH^3}(x_0,\gamma \cdot x_0)}, 
\]
where the implied constant depends on $\Lambda$ and $\Gamma_0$ only. Note that we have proved slightly more than the estimate. We have proved that for every $T\geq 1$, there is some set of elements $S_T\subset \Lambda$ so that for all $x\in K_0$:
\[
\sum_{\substack{\gamma\in\Lambda \text{ s.t.} \\ \dist_{\HH^3}(x,\gamma \cdot x) \leq T}} e^{-\dist_{\HH^3}(x,\gamma \cdot x)} \leq \sum_{\gamma \in S_T} e^{-\dist_{\HH^3}(x,\gamma x)}.
\]

All that remains is to deal with points $x\in\calF$ in the thin part of $\calF$. So we assume that $x\in \calF-K_0$ and we need to figure out which other contributions to the sum we get.

We first note that there cannot be any new hyperbolic contributions. Indeed, because $\eps_0 \leq \sys(\Gamma_0\backslash\HH^3)/2$, the axis of any hyperbolic element intersects $K_0$ and thus the contribution of any hyperbolic element is accounted for in the $\ll e^{(\delta-1+\eps)\cdot T}$ coming from the sum over the elements in $S_T$, and the individual contributions can only get smaller if we move away from the axis.

So we need to worry about parabolics. An elementary calculation shows that if $\gamma\in\Gamma_0$ is a parabolic element fixing any cusp $\mathbf{a}$ of $\Gamma$ and $t_{\mathbf{a}}(x') > t_{\mathbf{a}}(x)$, then
\[
\dist_{\HH^3}(x',\gamma \cdot x') <  \dist_{\HH^3}(x,\gamma \cdot x).
\]
In words: points higher up in the cusp get translated less far away. An explicit formula will appear in the computations below.

We note that, given a point $x\in \calF-K_0$, and a cusp $\mathbf{a}$ of $\Gamma_0$, there is a unique cusp $\mathbf{b}$ of $\Lambda$ that realizes 
\[
\max\st{ t_{\mathbf{b}}(x) }{ \mathbf{b}\sim_{\Gamma_0}\mathbf{a}}.
\]
Moreover, if there is some other cusps $\mathbf{b}'$ of $\Lambda$ such that $\mathbf{b}'\sim_{\Gamma_0}\mathbf{a}$ then there exists some point $x'\in K_0$ such that $t_{\mathbf{b}'}(x') > t_{\mathbf{b}'}(x)$. In particular, if $\gamma\in\Lambda$ is a parabolic that fixes this cusp $\mathbf{b}'$, we have
\[
\dist_{\HH^3}(x',\gamma\cdot x') < \dist_{\HH^3}(x,\gamma\cdot x)
\]
and so $ \dist_{\HH^3}(x,\gamma x)\leq T$. This means that the corresponding contribution is already covered by the $\ll e^{(\delta-1+\eps)\cdot T}$ coming from the sum over the elements in $S_T$.

The conclusion is that per cusp $\mathbf{a}$ of $\Gamma_0$, we only need to account for the elements of $\Lambda_{\mathbf{b}}$ for a single representative $\mathbf{b}$ of $\mathbf{a}$. We have already noted in Section \ref{sec_geometry_of_cusps} that the rank of $\Lambda_{\mathbf{b}}$ might be below two. In particular, if it vanishes, then there is no further contribution. So we only need to deal with the case in which it's of rank one or two. We will do the computation in the rank two case below, the computation in the rank one case is similar.

In this case, changing coordinates so that our cusp $\mathbf{a}$ is at infinity, we get that 
\[
\Lambda_{\mathbf{a}} = \langle (z,t)\mapsto (z+a,t),(z,t)\mapsto (z+b,t) \rangle
\]
for some $a,b\in\CC$ that are linearly independent over $\RR$. We have
\[
\dist_{\HH^3}\Big( (z,t),(z+m\cdot a + n\cdot b,t) \Big) = 2\sinh^{-1}\left(\frac{\abs{m\cdot a + n\cdot b}}{2t}\right).
\]
So 
\begin{multline*}
\sum_{\substack{\gamma \in \Lambda_{\mathbf{a}} \text{ s.t.}\\ \dist_{\HH^3}(x,\gamma \cdot x) \leq T}} e^{-\dist_{\HH^3}(x,\gamma\cdot x)} \ll \sum_{\substack{(m,n)\in\ZZ^2 \text{ s.t.} \\ \abs{m\cdot a + n \cdot b} \leq 2\cdot t(x) \cdot \sinh(T/2) }} \frac{t(x)^2}{\abs{m\cdot a+n\cdot b}^2}\\ \ll \frac{t(x)^2}{\mathrm{area}_\Lambda(\mathbf{a})} \log\left(\frac{t(x)^2\cdot \sinh(T/2)}{\mathrm{area}_{\Lambda}(\mathbf{a})}\right),
\end{multline*}
where we have used that $\area_{\Lambda}(\mathbf{a})=\abs{\det(a,b)}\cdot\area_{\Gamma_0}(\mathbf{a})$. 
\end{proof}

\subsection{Tangle-freeness}

Now we need a version tangle-freeness. Our random manifolds are not tangle-free in the usual sense: they do not look like $\HH^3$ on a local scale, but rather like the octa-tree. As such, we define:
\begin{dff}
Let $L>0$, let $\Gamma<\Isom(\HH^3)$ be discrete and let $\Lambda < \Gamma$. Then we say that $\Gamma$ is \textbf{$L$-$\Lambda$-tangle-free} at $x\in\Gamma\backslash\HH^3$ if for some lift $\widetilde{x}\in \HH^3$ (and hence for all lifts) of $x$, the set
\[
\Big\{ \gamma\in\Gamma; \; d(\widetilde{x},\gamma \cdot \widetilde{x}) \leq L \Big\} -
\Big\{ \gamma\in\Lambda; \; d(\widetilde{x},\gamma \cdot \widetilde{x}) \leq L \Big\}
\]
is either empty or generates a virtually abelian group. We say $\Gamma$ is $L$-$\Lambda$-tangle-free is it is $L$-$\Lambda$-tangle-free at all $x\in\Gamma\backslash\HH^3$.
\end{dff}

This version of tangle-freeness combined with the proposition above now implies a bound on $L^\infty$-norms of eigenfunctions. 
\begin{prp}\label{prp_tangle_free_deloc}
Let $\Gamma_0<\Isom(\HH^3)$ be a lattice and $\Lambda \triangleleft \Gamma_0$. Moreover suppose that $\Lambda<\Gamma<\Gamma_0$, where $\Gamma$ is some geometrically finite group and $\lambda < \lambda_0 = \lambda_0(\Lambda\backslash\HH^3)$. If $f:\Gamma\backslash\HH^3\to\RR$ is an $L^2$-normalized Laplacian eigenfunction of eigenvalue $\lambda$, then
\[
\abs{f(x)}^2 \ll_{\Gamma_0,\Lambda,\eps}  (1-\lambda) \cdot e^{-2L\cdot  \sqrt{1-\lambda}} \cdot \Big( e^{L\cdot (\sqrt{1-\lambda_0}+\eps )} +R_{\mathrm{cusp},\Gamma}(x,L)\Big) 
\]
uniformly for all $x\in \Gamma\backslash\HH^3$ at which $\Gamma$ is $L$-$\Lambda$-tangle-free and all groups $\Gamma$, where
\[
R_{\mathrm{cusp},\Gamma}(x,L) = \sum_{\substack{\mathbf{a} \text{ a cusp of }\Gamma_0 \\ \text{such that }t_{\mathbf{a}}(x) \geq 1 }} R_{\mathbf{a}}(x,L) +  \sum_{\substack{\mathbf{b} \text{ a cusp of }\Gamma \\ \text{such that }t_{\mathbf{a}}(x) \geq 1 }} R_{\mathbf{b},\Gamma}(x,L),
\]
$R_{\mathbf{a}}(x,L)$ is as in Lemma \ref{lem_kernel_growth} and
\[
R_{\mathbf{b},\Gamma}(x,L) = \frac{t_{\mathbf{b}}(x)^2}{\area_{\Gamma}(\mathbf{b})} \cdot \log\left(\frac{t_{\mathbf{b}}(x)\cdot \sinh(L/2)}{\mathrm{area}_{\Gamma}(\mathbf{b})} \right).
\]
\end{prp}

\begin{proof}
Using Lemma \ref{lem_kernel_growth} and that $\delta(\Lambda) = 1 + \sqrt{1-\lambda_0}$, the only thing that remains is to understand the contributions not coming from $\Lambda$. But these generate virtually abelian subgroups. So either they contain a hyperbolic element, in which case they are virtually cyclic and the fact that $\Gamma_0$ is a lattice implies that there is a uniform lower bound on the translation length and hence the contribution to the sum is at most linear in $L$, or they come from a parabolic subgroup (of rank $2$) and the contribution is at most $\ll  R_{\mathbf{b}}(x,L)$ like in the proof of Lemma \ref{lem_kernel_growth}.
\end{proof}

\section{Flattening an eigenfunction}

In this section we assume we are given an octahedral orbifold $M$. We will study the effect of flattening an eigenfunction near an interior face of this orbifold. 

One of the tools we will need is the Ismagilov--Morgan--Simo--Sigal localization formula \cite[Theorem 3.2]{FKS}:

\begin{thm}[Ismagilov--Morgan--Simo--Sigal]\label{thm_IMSS} Let $M$ be a Riemannian manifold and suppose $\{J_i\}_{i\in\calI}$ is a family of smooth functions $J_i:M\to [0,1]$ such that
\begin{enumerate}
\item $\sum_{i\in\calI} J_i^2 \equiv 1$,
\item on any compact subset $K\subset M$, only finitely many $J_i$ are non-zero, and
\item $\sup_{x\in M} \sum_{i\in\calI} \abs{\nabla J_i(x)}^2 < \infty$.
\end{enumerate}
Then 
\[
\Delta = \sum_{i\in\calI} J_i \Delta J_i - \sum_{i\in\calI} \abs{\nabla J_i}^2.
\]
\end{thm}

Another input is a $3$-dimensional version of a lemma due to Gamburd \cite[Lemma 4.1]{Gamburd}, which we prove with the same strategy as Gamburd:
\begin{lem}\label{lem_cusp_weight}
For any $\lambda_0<1$, there exists a constant $c_{\lambda_0}>0$ such that the following holds. Let $\Gamma < \Isom(\HH^3)$ be a geometrically finite subgroup and let $f:\Gamma\backslash\HH^3\to\RR$ be an $L^2$-normalized Laplacian eigenfunction of eigenvalue $\lambda<\lambda_0$ and for $T>0$, let 
\[
C_T = \st{ (x\cdot u+ y\cdot v, t) }{ x,y\in (0,1), t\geq T } \subset \HH^3
\]
for some $u,v\in\CC$ be a fundamental domain for a horoball in the cusps of $\Gamma$ (normalized so that the cusp lies at $\infty$). Then
\[
\frac{\int_{C_T} f^2 \; \mathrm{d}\vol_{\HH^3}}{\int_{C_{2^kT}} f^2 \; \mathrm{d}\vol_{\HH^3}} \geq (1+c_{\lambda_0})^k \quad \text{for all }k\geq 1
\]
\end{lem}

\begin{proof}
As we noted above, the proof uses the same idea as Gamburd's. We explain what needs to modified here. By \cite[Theorem 3.3.1]{ElstrodtGrunewaldMennicke}, we may write a Fourier expansion in the cusp for $f$ of the form
\[
f(z,t) = a_0\cdot  t^{1-s}+ \sum_{\mu \in \left(\ZZ u+\ZZ v\right)-\{0\}} a_{\mu} \cdot t \cdot K_s\Big(2\pi\cdot \abs{\mu}\cdot  t \Big) \cdot e^{2\pi i \langle \mu,z\rangle}
\]
where $\lambda = 1-s^2$ and $K_s$ is the same Bessel function as that defined in \cite[(4.4)]{Gamburd} (see also \cite[p. 66]{MagnusOberhettingerSoni}). So, Parseval's identity yields that
\[
\int_{C_R} \abs{f(z,t)}^2 \mathrm{d}\vol_{\HH^3} = \abs{a_0}^2 \int_R^\infty t^{-1-2s} \mathrm{d}t + \sum_{\mu \in \left(\ZZ u+\ZZ v\right)-\{0\}} \abs{a_{\mu}}^2 \int_R^\infty  \frac{K_s\Big(2\pi\cdot \abs{\mu}\cdot  t \Big)^2}{t} \mathrm{d}t.
\]
We have
\[
\frac{\int_{R}^{2R} t^{-1-2s} \mathrm{d}t}{\int_{2R}^\infty t^{-1-2s} \mathrm{d}t} = 2^{2s}-1 > 2^{2s_0}-1 >0,
\]
where $s_0>0$ is such that $\lambda_0 = 1-s_0^2$.

Moreover, we're lucky and the other integrals that appear on the right hand side are exactly the same as those that appear in Gamburd's proof. He proved\footnote{In fact, Gamburd states the result for $s\in (s_0,\frac{1}{2})$ but the exact same proof works for $s\in (s_0,1)$. The constant $c_{s_0}$ might however get slightly worse.} that there exists a constant $c_{s_0}$ such that for all $s\in (s_0,1)$:
\[
 \int_R^{2R}  \frac{K_s\Big(2\pi\cdot \abs{\mu}\cdot  t \Big)^2}{t} \mathrm{d}t \quad \Bigg/ \quad  \int_{2R}^\infty  \frac{K_s\Big(2\pi\cdot \abs{\mu}\cdot  t \Big)^2}{t} \mathrm{d}t \quad \geq \quad c_{s_0}.
\]

As such, we obtain that
\[
\frac{\int_{C_T-C_{2T}} f^2 \; \mathrm{d}\vol_{\HH^3}}{\int_{C_{2T}} f^2 \; \mathrm{d}\vol_{\HH^3}} > c_{\lambda_0},
\]
which gives us the case $k=1$. The remainder we will prove by recursion. Indeed, using the induction hypothesis that
\[
\frac{\int_{C_T-C_{2^mT}} f^2 \; \mathrm{d}\vol_{\HH^3}}{\int_{C_{2^mT}} f^2 \; \mathrm{d}\vol_{\HH^3}} \geq (1+c_{\lambda_0})^m-1 \quad \text{for all }1\leq m \leq k,
\]
we obtain
\begin{align*}
\int_{C_T-C_{2^{k+1}T}} f^2 \; \mathrm{d}\vol_{\HH^3} & = 
\int_{C_T-C_{2^kT}} f^2 \; \mathrm{d}\vol_{\HH^3}  + \int_{C_{2^{k}T}-C_{2^{k+1}T}} f^2 \; \mathrm{d}\vol_{\HH^3} \\
& \geq \Big((1+c_{\lambda_0})^k-1\Big) \cdot \int_{C_{2^kT}} f^2\; \mathrm{d}\vol_{\HH^3} + c_{\lambda_0} \int_{C_{2^{k+1}T}}  f^2\; \mathrm{d}\vol_{\HH^3}  \\
& =  \Big((1+c_{\lambda_0})^k-1+c_{\lambda_0}\Big) \cdot \int_{C_{2^{k+1}T}} f^2\; \mathrm{d}\vol_{\HH^3} \\
& \quad + \Big((1+c_{\lambda_0})^k-1\Big) \cdot \int_{C_{2^kT}-C_{2^{k+1}T}} f^2\; \mathrm{d}\vol_{\HH^3} \\
& \geq \Big((1+c_{\lambda_0})^{k+1}-1\Big) \cdot \int_{C_{2^{k+1}T}} f^2\; \mathrm{d}\vol_{\HH^3},
\end{align*}
which, after rearranging, gives us the lemma.
\end{proof}

We will now prove:
\begin{prp}
Fix $\lambda_0<1$ and let $D$ denote an ideal hyperbolic triangle. There exists a function $\beta:D\to\RR$ and a constant $B>0$ such that the following holds. Let $X=\Gamma\backslash \HH^3$ be an octahedral orbifold that is $L$-$\Gamma_\infty$-tangle-free, and let $F_1,\ldots,F_k\subset X$ be interior faces (that we will identify with copies of $D$ in $X$). Moreover let $f:X\to\RR$ be an $L^2$-normalized Laplacian eigenfunction of eigenvalue $\lambda \in (0,\lambda_0)$. Then there exists a function $f':X\to\RR$ such that
\begin{itemize}
\item $f'\rvert_{F_i}-\beta$ is constant and this constant does not depend on $i$, for $1\leq i \leq k$,
\item $f'$ has zero average on $X$, and
\item we have that
\[
\frac{\langle f',\Delta f'\rangle}{\langle f',f'\rangle} - \lambda \leq e^{-C_{\lambda_0}\cdot L},
\]
for some constant $C_{\lambda_0}>0$ that depends on $\lambda_0$ only.
\end{itemize}
\end{prp}

\begin{proof}
We first once and for all fix a smooth function $J:\RR\to [0,1]$ such that
\[
J(x) = \left\{\begin{array}{cc}
0 & \text{if } \abs{x}\leq 1 \\
1 & \text{if } \abs{x} \geq 2
\end{array}\right..
\]

We will build our new function $f'$ in three steps. First we flatten out $f$ in the cusps that intersect our faces, then we flatten the result out near the required faces and finally we will remove the average of the function from it, to make it orthogonal to constant functions.

We first compute what happens to $\langle f,\Delta f\rangle$. We'll start with the cusps. Every face $F_i$ is incident to three cusps. Let $\mathbf{a}$ be one of these cusps, which we will normalize to lie at $\infty$ (again using our fixed Möbius transformation $h_{\mathbf{a}}$ from Section \ref{sec_geometry_of_cusps}). We will choose a height $\tau_{\mathbf{a}}\geq 1$ above which we will flatten out $f$. This height will depend on the parameter $L$. How it does, we will determine below. 

Now set: 
\[
C_{\tau_{\mathbf{a}}} = \st{x\in \calF}{t_{\mathbf{a}}(x) \geq \tau_{\mathbf{a}}},
\]
where $\calF$ is a fundamental domain for $\Gamma$ chosen so that $\mathbf{a}$ indeed lies at $\infty$.

For $x\in C_{\tau_{\mathbf{a}}}$, we define 
\[
J_{\mathrm{cusp}}(x) = 1-J(t_{\mathbf{a}}(x)/\tau_{\mathbf{a}}),
\]
which allows us to define
\[
f_1(x) = \left\{
\begin{array}{ll}
f(x) \cdot J_{\mathrm{cusp}}(x) & \text{if } t_{\mathbf{a}}(x) \geq \tau_{\mathbf{a}} \\
f(x) & \text{otherwise}.
\end{array}\right.
\]
Theorem \ref{thm_IMSS} now implies that
\begin{align*}
\langle f_1 , \Delta f_1 \rangle - \langle f , \Delta f \rangle & = \int_{C_{\tau_{\mathbf{a}}}}  \Big(J_{\mathrm{cusp}}(x) \Delta \Big(  J_{\mathrm{cusp}}(x) \cdot f(x)\Big) - \Delta f(x)\Big)\cdot f(x) \; \mathrm{d}\vol_{\HH^3}(x) \\
& = \int_{C_{\tau_{\mathbf{a}}}} \abs{\nabla J_{\mathrm{cusp}}(x)}^2 \cdot f(x)^2 + \abs{\nabla \sqrt{1-J_{\mathrm{cusp}}^2(x)}}^2 \cdot f(x)^2 \; \mathrm{d}\vol_{\HH^3}(x) \\
& \quad - \int_{C_{\tau_{\mathbf{a}}}}  \Delta \Big(\sqrt{1-J_{\mathrm{cusp}}(x)^2} f(x) \Big) \cdot \sqrt{1-J_{\mathrm{cusp}}(x)^2} \cdot f(x) \; \mathrm{d}\vol_{\HH^3}(x).
\end{align*}
Now we use the fact that $\Delta$ is a positive operator, and obtain that
\begin{align*}
\langle f_1 , \Delta f_1 \rangle - \langle f , \Delta f \rangle & = \int_{C_{\tau_{\mathbf{a}}}}  \Big(J_{\mathrm{cusp}}(x) \Delta \Big(  J_{\mathrm{cusp}}(x) \cdot f(x)\Big) - \Delta f(x)\Big)\cdot f(x) \; \mathrm{d}\vol_{\HH^3}(x) \\
& \leq \int_{C_{\tau_{\mathbf{a}}}} \abs{\nabla J_{\mathrm{cusp}}(x)}^2 \cdot f(x)^2 + \abs{\nabla \sqrt{1-J_{\mathrm{cusp}}^2(x)}}^2 \cdot f(x)^2 \; \mathrm{d}\vol_{\HH^3}(x) 
\end{align*}
Both the gradients in the integral on the right hand side above are supported in 
\[
\st{x\in C_{\tau_{\mathbf{a}}}}{t_{\mathbf{a}}(x) \in [\tau_{\mathbf{a}},2\tau_{\mathbf{a}}] }
\]
This means that 
\[
\langle f_1 , \Delta f_1 \rangle - \langle f , \Delta f \rangle \leq  \norm{f\vert_{t_{\mathbf{a}}(x) \in [\tau_{\mathbf{a}}, 2\tau_{\mathbf{a}}]}}_\infty^2 \cdot \int_{C_{\tau_{\mathbf{a}}}} \abs{\nabla J_{\mathrm{cusp}}(x)}^2  + \abs{\nabla \sqrt{1-J_{\mathrm{cusp}}^2(x)}}^2  \; \mathrm{d}\vol_{\HH^3}(x).
\]
Now we pick a fundamental parallelogram $\calD$ for the action of $\Gamma_{\mathbf{a}}$ on $\CC$ and write
\begin{align*}
\int_{C_{\tau_{\mathbf{a}}}} \abs{\nabla J_{\mathrm{cusp}}(x)}^2 \; \mathrm{d}\vol_{\HH^3}(x)
& =
 \int_{\calD} \int_{h_\mathbf{a}}^{2\tau_{\mathbf{a}}}\frac{1}{t} \left( \frac{\partial}{\partial t}J\left(\frac{t}{\tau_{\mathbf{a}}}\right)\right)^2 \;\mathrm{d}t\;\mathrm{d}x \;\mathrm{d}y \\
& =
 \frac{1}{\tau_{\mathbf{a}}^2} \int_{\calD} \int_1^2\frac{1}{u} \left( \frac{\partial}{\partial u}J(u)\right)^2 \;\mathrm{d}u\;\mathrm{d}x \;\mathrm{d}y \\ 
& = \frac{\area_{\Gamma}(\mathbf{a})}{\tau_{\mathbf{a}}^2}  \cdot C,
\end{align*}
and likewise
\[
\int_{C_{\tau_{\mathbf{a}}}} \abs{\nabla \sqrt{1-J_{\mathrm{cusp}}(x)^2}}^2 \; \mathrm{d}\vol_{\HH^3}(x) \leq \frac{\area_{\Gamma}(\mathbf{a})}{\tau_{\mathbf{a}}^2}  \cdot C.
\]
where $C>0$ is some constant that depends on our choice of $J$ only. Plugging this back into our inequality and using Proposition \ref{prp_tangle_free_deloc}, we obtain $ \langle f_1,\Delta f_1\rangle - \langle f,\Delta f\rangle \ll  E_1$ where
\begin{equation}\label{eq_E1}
E_1(i) =\sum_{\substack{\mathbf{a} \text{ cusp of }\Gamma \\ \text{incident to }F_i}} e^{-2L\cdot \sqrt{1-\lambda}}\cdot \left( \frac{\area_\Gamma(\mathbf{a})}{\tau_{\mathbf{a}}^2} \cdot e^{L\cdot (\sqrt{1-\lambda_0}+\eps)} +  \log\Big(\tau_{\mathbf{a}}\cdot\sinh(L/2)\Big) \right),
\end{equation}
where we have again (like in the proof of Lemma \ref{lem_kernel_growth}) used the fact that $t_{\mathbf{a}}$ is large for at most one cusp $\mathbf{a}$ of $\Gamma$ at a time. This error will  need to be summed over all the faces.

Once this has been done for all the cusps adjacent to the interior faces $F_1,\ldots F_k$, we obtain a function that we will still call $f_1:X\to\RR$. Our next step is to flatten the resulting function near the faces $F_i$ in the thick part of $X$. To make this precise, we use $12$ copies of the fundamental corner $\calC$ to parametrize a neighborhood of our face $F_i$. In the coordinates of $\calC$, we will write
\[
J_{\mathrm{face}}(x,y,t) = J(4x) \quad \text{for } (x,y,t)\in\calC
\]
and define
\[
f_2(x) = f_1(x) \cdot J_{\mathrm{face}}(x) \quad \text{for } x\in\calC.
\]
Writing $\calC_1,\ldots\calC_{12}$ for the copies of the fundamental corner adjacent to $F_i$, we may again apply Theorem \ref{thm_IMSS} to obtain
\begin{align*}
\langle f_2,\Delta f_2\rangle - \langle f_1,\Delta f_1\rangle & = \sum_{i=1}^{12} \int_{\calC_i} \abs{\nabla J_{\mathrm{face}}(x)}^2\cdot f_1(x)^2 + \abs{\nabla \sqrt{1-J_{\mathrm{face}}(x)^2}}^2 \cdot f_1(x)^2 \; \mathrm{d}\vol_{\HH^3}(x) \\
 & \quad +  \int_{\calC_i} \Delta \left(\sqrt{1-J_{\mathrm{face}}(x)^2}\cdot f_1(x)\right) \cdot \sqrt{1-J_{\mathrm{face}}(x)^2} \cdot f_1(x)\; \mathrm{d}\vol_{\HH^3}(x) \\
&  \leq \sum_{i=1}^{12} \int_{\calC_i} \abs{\nabla J_{\mathrm{face}}(x)}^2\cdot f_1(x)^2 + \abs{\nabla \sqrt{1-J_{\mathrm{face}}(x)^2}}^2 \cdot f_1(x)^2\;\mathrm{d}\vol_{\HH^3}(x) \\
& \leq C \cdot \norm{f\vert_{K_{\tau}}}_\infty^2
\end{align*}
where $C>0$ is a constant depending only on our choice of the function $J$ (the volume that enters into it can be bounded by the volume of $12$ fundamental corners). Here we have used the fact that $\Delta$ is a positive operator in the second step, the fact that $\abs{f_1}\leq \abs{f}$ and the fact that $f_1$ is supported in $K_h$, where 
\[
K_{\tau} = \st{x\in X}{t_{\mathbf{a}}(x) \leq 2 \tau_{\mathbf{a}} \text{ for all cusps } \mathbf{a} \text{ of }\Gamma}.
\]
So using Proposition \ref{prp_tangle_free_deloc} again, we obtain that $
\langle f_2,\Delta f_2\rangle - \langle f_1,\Delta f_1\rangle \ll E_2$ with 
\begin{equation}\label{eq_E2}
E_2(i) = e^{-2L\cdot \sqrt{1-\lambda}}\cdot\left( e^{L\cdot(\sqrt{1-\lambda_0}+\eps)} + \sum_{\substack{\mathbf{a} \text{ a cusp of }\Gamma \\ \text{incident to }F_i}} \frac{\tau_{\mathbf{a}}}{\area_{\Gamma}(\mathbf{a})} \cdot \log\left( \frac{\tau_{\mathbf{a}}\cdot \sinh(L/2)}{\mathrm{area}_{\Gamma}(\mathbf{a})}\right)\right),
\end{equation} 
which we again need to sum over all the faces when we redefine $f_2$ to be the function obtained from applying the above to all the faces $F_i$.

Finally, we need to remove the average from the function $f_2$ in order to make it orthogonal to constant functions. So we set 
\[
f'(x) = f_2(x) - \frac{1}{\vol(X)}\int_X f_2(y)\;\mathrm{d}\vol(y).
\]
Because the Laplacian annihilates constant functions, we obtain that 
\begin{equation}\label{eq_fDfbound}
\langle f',\Delta f'\rangle = \langle f_2, \Delta f_2\rangle \ll \langle f,\Delta f \rangle + \sum_{i=1}^k E_1(i) + E_2(i).
\end{equation}
Now we need to understand how $\norm{f'}^2$ compares to $\norm{f}^2$. We have
\begin{align*}
\langle f',f'\rangle & = \langle f_2,f_2 \rangle - \frac{1}{\vol(X)} \left(\int_X f_2(y)\;\mathrm{d}\vol(y)\right)^2 \\
& = \langle f,f \rangle +2 \langle f_2-f,f\rangle + \langle f_2-f,f_2-f\rangle - \frac{1}{\vol(X)} \left(\int_X f_2(y)\;\mathrm{d}\vol(y)\right)^2.
\end{align*}
If we use that $\int_X f(y) \;\mathrm{d}\vol(y)=0$, we obtain, using Cauchy--Schwartz,
\begin{align*}
\frac{1}{\vol(X)} \left(\int_X f_2(y)\;\mathrm{d}\vol(y)\right)^2 
& =  
\frac{1}{\vol(X)} \left(\int_X f_2(y)-f(y)\;\mathrm{d}\vol(y)\right)^2 \\
& \leq \langle f_2-f,f_2-f \rangle.
\end{align*}
This means that, again using Cauchy--Schwartz and the fact that $\norm{f}=1$,
\[
\langle f',f' \rangle \geq \langle f,f\rangle - 2\cdot \norm{f_2-f}^2.
\]
Now we have run out generalities and we need to use the estimates we have. We have
\begin{align*}
\norm{f_2-f}^2 = \sum_{i=1}^k \sum_{j=1}^{12} \int_{\calC_{ij}} \left(1-J_{\mathrm{face}}(x)\cdot J_{\mathrm{cusp}}(x)\right)^2 f(x)^2 \mathrm{d}\vol(x) \\
\quad + \sum_{\substack{\mathbf{a} \text{ a cusp of }\Gamma \\  \text{incident to } F_i }} \int_{C_{\tau_{\mathbf{a}}}} (1-J_{\mathrm{cusp}}(x))^2\cdot f(x)^2 \; \mathrm{d}\vol(x),
\end{align*}
where $\calC_{ij}$ for $j=1,\ldots,12$ are the copies of the fundamental corner incident to the face $F$. Also note that the first term abuses notation a bit. $J_{\mathrm{cusp}}(x)$ is only defined for $x$ in a horoball. What we mean here is that it is the same function extended to $1$ in the region where it was not previously defined. Note that by construction (the fact that $\tau_{\mathbf{a}}\geq 1$ for all cusps $\mathbf{a}$) the cusp neighborhoods on which the functions $J_{\mathrm{cusp}}$ are defined don't intersect. So, given $x\in\calC_{ij}$, there is no ambiguity as to with respect to which cusp the function $J_{\mathrm{cusp}}(x)$ is defined. We beg the reader who has made it this far into the text for forgiveness. In any event, we get rid of these factors immediately and just write
\[
\norm{f_2-f}^2 \leq \sum_{i=1}^k \sum_{j=1}^{12} \int_{\calC_{ij}} f(x)^2 \mathrm{d}\vol(x)  + \sum_{\substack{\mathbf{a} \text{ a cusp of }\Gamma \\  \text{incident to } F_i }} \int_{C_{\tau_{\mathbf{a}}}}  f(x)^2 \; \mathrm{d}\vol(x),
\]

We will now start with the first type of integrals appearing above. We note that the fundamental corner $\calC_{ij}$ only interacts with one of the cusps of $\Gamma$, that we will call $\mathbf{a}_{ij}$ in what follows. We have
\[
\int_{\calC_{ij}} f(x)^2\; \mathrm{d}\vol(x) = \int_{\calC_{ij}\cap C_{\tau_{\mathbf{a}}}} f(x)^2\; \mathrm{d}\vol(x) +  \int_{\calC_{ij}- C_{\tau_{\mathbf{a}}}} f(x)^2\; \mathrm{d}\vol(x).
\]
For the first of these, we use Lemma \ref{lem_cusp_weight} and the fact that $\norm{f}=1$ to conclude that
\[
\int_{\calC_{ij}\cap C_{\tau_{\mathbf{a}}}} f(x)^2\; \mathrm{d}\vol(x) \leq \frac{1}{\tau_{\mathbf{a}}^{C_{\lambda_0}}}
\]
for some uniform constant $C_{\lambda_0}>0$ depending on $\lambda_0$ only. For the second integral we use Proposition \ref{prp_tangle_free_deloc} again to obtain
\begin{multline*}
  \int_{\calC_{ij}- C_{\tau_{\mathbf{a}_{ij}}}} f(x)^2\; \mathrm{d}\vol(x) \ll \norm{f\vert_{\calC_{ij}- C_{\tau_{\mathbf{a}_{ij}}}}}_{\infty}^2 \\
    \ll e^{-4L\cdot \sqrt{1-\lambda}}\cdot \left(e^{L\cdot (\sqrt{1-\lambda_0}+\eps)}+\frac{\tau_{\mathbf{a}_{ij}}}{\area_\Gamma(\mathbf{a}_{ij})}\cdot \log\left(\frac{\tau_{\mathbf{a}_{ij}}\cdot\sinh(L/2)}{\area_{\Gamma}(\mathbf{a}_{ij})} \right)\right)^2.
\end{multline*}
The second type of integrals are now also already covered. We have, using Lemma \ref{lem_cusp_weight},
\[
\int_{C_{\tau_{\mathbf{a}}}} f(x)^2\mathrm{d}\vol(x) \leq \frac{1}{\tau_{\mathbf{a}}^{C_{\lambda_0}}}.
\]
Putting the above together and using the fact that the involved cusp neighbourhoods don't intersect, we get that 
\begin{equation}\label{eq_ffbound}
\langle f',f'\rangle \gg \langle f,f\rangle - \sum_{i=1}^k E_3(i)
\end{equation}
with
\begin{equation}\label{eq_E3}
E_3(i) = \sum_{\substack{\mathbf{a} \text{ a cusp} \\
\text{of }\Gamma \text{ in-} \\ \text{cident} \\ \text{to }F_i}} \frac{1}{\tau_{\mathbf{a}}^{C_{\lambda_0}}} + e^{-2L\cdot\sqrt{1-\lambda}} \cdot \left( e^{L\cdot (\sqrt{1-\lambda_0}+\eps)} + \frac{\tau_{\mathbf{a}}}{\area_\Gamma(\mathrm{a})}\cdot \log\left(\frac{\tau_{\mathbf{a}}\cdot\sinh(L/2)}{\area_\Gamma(\mathbf{a})}\right) \right).
\end{equation}

Our function $f'$ satisfies the first two conditions of the statement. Now looking at equations \eqref{eq_E1}, \eqref{eq_E2} and \eqref{eq_E3}, we see that if we set
\[
\tau_{\mathbf{a}} = \max\Big\{\sqrt{\area_\Gamma(\mathbf{a})},e^{\sqrt{1-\lambda_0}\cdot L}\Big\},
\]
we obtain that $\langle f',f'\rangle >0$ and, filling in \eqref{eq_fDfbound} and \eqref{eq_ffbound},
\[
\frac{\langle f', \Delta f'\rangle}{\langle f',f'\rangle} \ll \frac{\lambda + k\cdot e^{-C_{\lambda_0}\cdot L}}{1-k\cdot e^{-C_{\lambda_0}\cdot L}} \ll \lambda + k\cdot e^{-C_{\lambda_0}\cdot L},
\]
for some constant $C_{\lambda_0}>0$ that depends on $\lambda_0$ only.
\end{proof}

\section{The bass note spectrum of the set of hyperbolic $3$-orbifolds}

In this section, we use our model to produce many different spectral gaps. That is, we prove Theorem \ref{thm_main3}, which follows directly from the next result:

\begin{thm}
For every $\eta>0$ and $\lambda \in [0,\lambda_0(T)]$ there exists a subgroup $\Gamma<\Gamma_{\mathrm{SA}}$ of finite index such that
\[
\abs{\lambda_1(\Gamma\backslash\HH^3) - \lambda} < \eta.
\]
\end{thm}

\begin{proof} What we will actually do is, for $\eta>0$ fixed, produce a finite index subgroup $\Gamma'<\Gamma_{\mathrm{SA}}$ such that, such that the first eigenvalues of the index two subgroups of $\Gamma'$ are $\eta$-dense in $[0,\lambda_0(T)]$. The idea is to first produce subgroups whose spectral gaps lie at the ends of the interval and then produce the intermediate values using a process we will call switching between covers of degree two.

So, given $\Gamma<\Gamma_{\mathrm{SA}}$ of finite index, we consider the set of degree two covers
\[
\calH_2(\Gamma) = \st{ Y \to \Gamma\backslash\HH^3 }{ Y \text{ a degree }2\text{ cover}}.
\]
This set crucially includes the cover $Y_{\mathrm{disconnected}}\to\Gamma\backslash\HH^3$ by two copies of $\Gamma\backslash\HH^3$, which has the property that 
\[
\lambda_1\left(Y_{\mathrm{disconnected}}\right) = 0,
\]
which thus gives us the lower end of the interval, independently of the subgroup $\Gamma<\Gamma_{\mathrm{SA}}$.

In order to produce a spectral gap at the higher end of the interval, we use methods similar to those in Section \ref{sec_near_optimal}. We already know that we can produce subgroups $\Gamma<\Gamma_{\mathrm{SA}}$ of finite index whose spectral gap is at least $\lambda_0(\Gamma_\infty\backslash\HH^3)-\eta/2$. Now we need to find a subgroup of index two of such a group. Recall that $\Gamma_{\mathrm{SA}}$ itself contains the congruence group $\Gamma(2)$ of $\PGL(2,\ZZ[i])$ as an index two subgroup. First we note that by Propositions \ref{prp_main} and \ref{prp_eigenvalue_gamma_2},
\[
\inf\Big( \sigma\left(\Delta_{\Gamma(2)\backslash\HH^3}\right)-\{0\} \Big) = 1  > \lambda_0(T).
\]
Moreover, $\Gamma(2) = \ker(\mathrm{or}:\Gamma_{\mathrm{SA}}\to\ZZ/2\ZZ)$, where $\mathrm{or}$ denotes the orientation homomorphism. The infilonian $\Gamma_\infty=\ker(\Gamma_{\mathrm{SA}}\to (\ZZ/2\ZZ)^{\ast 4})$ is contained in all our random subgroups $\Gamma_n$. Moreover, $\mathrm{or}:\Gamma_\infty\to\ZZ/2\ZZ$ is surjective because $\Gamma_\infty$ contains four of the reflections in the faces of $O$. This means that, if $\Gamma_n<\Gamma_{\mathrm{SA}}$ is a random subgroup of index $n$ whose monodromy factors through the map $\Gamma_{\mathrm{SA}}\to (\ZZ/2\ZZ)^{\ast 4}$, then $\mathrm{or}:\Gamma_n\to\ZZ/2\ZZ$ is surjective and hence that $[\Gamma_n:\Gamma_n\cap\Gamma(2)] = 2$.

In fact, the group $\Gamma_n\cap\Gamma(2)$ is exactly the Kleinian group corresponding to the double $DM_n$ of $M_n$ described above. So we already know from Theorem \ref{thm_main2} that
\[
\lambda_1\Big((\Gamma\cap\Gamma(2)) \backslash \HH^3\Big) > \lambda_0(\Gamma_\infty\backslash\HH^3) - \eta/2
\] 
thus giving us the other end of the interval.

Given a degree two cover $Y\to \Gamma_n\backslash\HH^3$ we say that $Y'\to\Gamma_n\backslash\HH^3$ is obtained from $Y$ by a \textbf{simple switching} if there exists an interior face $F$ of $\Gamma_n\backslash\HH^3$ such that $Y'$ is obtained from $Y$ by cutting $Y$ open along the pre-images $F_1$ and $F_2$ of $F$ under the covering map and regluing them. 

Any pair of elements in $\calH_2(\Gamma_n)$ can be connected by a finite sequence of simple switchings. Indeed, as manifolds with boundary, our orbifolds are homotopy equivalent to the graphs $G_n$ dual to the decomposition of $\Gamma_n\backslash\HH^3$ into copies of $O$. This means that degree two covers of $\Gamma_n\backslash\HH^3$  correspond one-to-one with degree two covers of $G_n$. We can move between degree two covers of $G_n$ by changing the cover near the two lifts of an edge (in graph theory, this gives the correspondence between two-covers and signings). Because of the duality, this process of changing the cover along the pre-images of an edge corresponds to cutting the cover $Y$ open along the two pre-images of an interior face of $\Gamma_n\backslash\HH^3$ and regluing the resulting faces in the opposite way.

Now we want to apply the machinery from the previous sections to prove that a simple switching doesn't change the spectral gap by too much, so we need some tangle-freeness. By \cite[Lemma 23]{BordenaveCollins}, for $T < (1/4-\eps)\cdot \log_3(n)$, the Schreier graph associated to a random $\varphi_n \in \Hom((\ZZ/2\ZZ)^{\ast 4},\sym_{2n})$ is $T$-tangle-free with high probability. This Schreier graph is exactly the dual graph $G_n$ mentioned in the previous paragraph. We claim that this implies that, for any fixed $L>0$, the corresponding random covers $\Gamma_n\backslash\HH^3$ of the super Apollonian orbifold $\Gamma_{\mathrm{SA}}\backslash\HH^3$ are $L$-$\Gamma_\infty$-tangle-free with high probability. This in turn implies the same for their degree two covers. 

To see the tangle-freeness of $\Gamma_n\backslash\HH^3$, we argue as follows. First fix some $\eps_0=\eps_0(L)>0$ below the Margulis constant of $\HH^3$ and let $K_n\subset \Gamma_n\backslash\HH^3$ denote the lift of the $\eps_0$-thick part of $\Gamma_{\mathrm{SA}}\backslash\HH^3$. We choose $\eps_0$ small enough so that closed geodesics on $\Gamma_{\mathrm{SA}}\backslash\HH^3$ of length $\leq L$ lie at distance at least $L$ from the $\eps_0$-thin part of $\Gamma_{\mathrm{SA}}\backslash\HH^3$.

Now, by construction, if a point $\widetilde{x}\in \HH^3$ projects to a point $x\in \Gamma_\infty\backslash\HH^3-K_n$, the only translates of $x$ at distance at most $L$ from $x$ (that don't come from parabolics in $\Gamma_\infty$) generate a cyclic subgroup. So $\Gamma_n$ is $L$-$\Gamma_\infty$-tangle-free at $x$.

For a point $x\in K_n$, we observe that, once $L$ has been fixed, $\eps_0$ is uniform. This means that $x$ lies at bounded distance from the dual graph $G_n$ to the decomposition of $\Gamma_n\backslash\HH^3$. So in particular, there cannot be two elements of $\Gamma_n-\Gamma_\infty$ that generate a non-elementary subgroup of $\Gamma_n$ and both translate a lift $\widetilde{x}\in\HH^3$ of $x$ by distance at most $L$. Indeed, the corresponding loops would lie at uniformly bounded distance from the dual graph $G_n$ and hence create a tangle, which is forbidden (once $n$ is large enough) by the result of Bordenave--Collins. We note that with an argument like in \cite[Section 5]{Magee_bassnotes}, this can be made much more effective and one would get logarithmic tangle-freeness. For our purposes, uniform tangle-freeness suffices.

Indeed, given a degree two cover $Y\to\Gamma\backslash\HH^3$, an interior face $F$ of $\Gamma_n\backslash\HH^3$ and a Laplacian eigenfunction $f:Y\to\RR$ of eigenvalue $\lambda_1(Y)$, we can produce a function $f':Y\to\RR$ that has looks the same on the pre-images $F_1$ and $F_2$ of $F$ and whose Rayleigh quotient is at most $\lambda_1(Y) + \eps(L)$, where $\eps(L)\to 0$ as $L\to\infty$. As such, $f'$ defines a function on the cover $Y'$ obtained from the simple switching in $F$ as well. Which means that $\lambda_1(Y') < \lambda_1(Y) + \eps(L)$. Doing the argument in reverse gives us a bound in the opposite direction. So, by choosing $n$ large enough, we obtain that for any $Y,Y'\in\calH_2(\Gamma)$ that differ by a simple switching
\[
\abs{\lambda_1(Y) - \lambda_1(Y')} < \eta,
\]
which proves our theorem for finite volume orbifolds. Moreover, our orbifolds are arithmetic, so it also proves the arithmetic version.
\end{proof}

\bibliography{bib}
\bibliographystyle{alpha}

\end{document}